\documentclass[11pt,oneside,english]{amsart}
\usepackage[T1]{fontenc}
\usepackage[latin9]{inputenc}
\usepackage{textcomp}
\usepackage{amsbsy}
\usepackage{amstext}
\usepackage{amsthm}
\usepackage{amssymb}
\usepackage{graphicx}

\numberwithin{equation}{section}
\numberwithin{figure}{section}
  \theoremstyle{plain}
  \newtheorem*{thm*}{\protect\theoremname}
  \theoremstyle{plain}
  \newtheorem*{prop*}{\protect\propositionname}
\theoremstyle{plain}
\newtheorem{thm}{\protect\theoremname}
  \theoremstyle{definition}
  \newtheorem{defn}[thm]{\protect\definitionname}
  \theoremstyle{plain}
  \newtheorem{prop}[thm]{\protect\propositionname}
  \theoremstyle{definition}
  \newtheorem{example}[thm]{\protect\examplename}
  \theoremstyle{remark}
  \newtheorem{rem}[thm]{\protect\remarkname}
  \theoremstyle{plain}
  \newtheorem{lem}[thm]{\protect\lemmaname}
  \theoremstyle{plain}
  \newtheorem{cor}[thm]{\protect\corollaryname}

\usepackage{babel}
  \providecommand{\corollaryname}{Corollary}
  \providecommand{\definitionname}{Definition}
  \providecommand{\examplename}{Example}
  \providecommand{\lemmaname}{Lemma}
  \providecommand{\propositionname}{Proposition}
  \providecommand{\remarkname}{Remark}
  \providecommand{\theoremname}{Theorem}

\begin{document}

\title[Phase Connectedness of Area Minimizing Partitionings]{On the Phase Connectedness of the Volume-Constrained Area Minimizing
Partitioning Problem}

\author{A. C. Faliagas}

\address{Department of Mathematics, University of Athens, Panepistemiopolis,
15784 Athens, Greece}

\email{afaliaga@math.uoa.gr}
\begin{abstract}
We study the stability of partitions in convex domains involving simultaneous
coexistence of three phases, viz. triple junctions. We present a careful
derivation of the formula for the second variation of area, written
in a suitable form with particular attention to boundary and spine
terms, and prove, in contrast to the two phase case, the existence
of stable partitions involving a disconnected phase.
\end{abstract}

\maketitle

\section{\label{sec:Intro}Introduction}

The phase partitioning problem involves the splitting of a domain
$\Omega\subset\mathbb{R}^{n}$ into a prescribed number of subsets,
the \emph{phases}, with the measure of each phase fixed, and minimality
of their perimeter in the interior of $\Omega$. Investigation of
interfaces and related phenomena started in the 19th century when
Plateau\cite{key-2} observed that soap films and bubble clusters
consisted of (a) smooth surfaces, (b) curves (liquid lines) along
which triples of surfaces met at equal angles, and (c) isolated points
where four such triple junctions met at equal angles. Early studies
of the mathematical problem of partitioning include Nitsche's paper\cite{key-15,key-16},
and Almgren's Memoir\cite{key-17}. White\cite{key-18} proved existence
and discussed regularity of equilibrium immiscible fluid configurations
using Flemming's flat chains\cite{key-106}. Taylor\cite{key-3,key-4}
characterized the minimal cones in $\mathbb{R}^{3}$. The regularity
of the liquid line was established by Kinderlehrer, Nirenberg and
Spruck\cite{key-19}.

The structure of the singular set of hypersurfaces and their clusters
was studied by using mean curvature flow methods\cite{key-107}. A
hypersurface evolves by mean curvature flow when the velocity is given
by the mean curvature vector. Volume preserving mean curvature flows
were used for the investigation of the dynamics of phase partitioning
problems with a volume constraint\cite{key-108,key-109}. These methods
apply to two-phase problems. For the three-phase partitions with prescribed
boundaries and triple junction topologies, the required constraints
render the formulation and the handling of the problem prohibitively
complex. Note that a pure mean curvature motion with nontrivial velocity
is not possible near the line where the three surfaces of a triple
junction intersect (the spine), see Section \ref{sec:App-TJ-R3},
text preceding equation (\ref{eq:A-38}). We found the direct variational
methods used in this work, which allow tangential variations besides
the usual normal ones, more convenient and suitable for the investigation
of the stability of multiphase problems involving multijunctions.

The problem of the phase connectedness was addressed by Sternberg
and Zumbrun (SZ)\cite{key-7}, for the particular case of two-phase
partitioning. They proved that \emph{stable two phase partitions in
strictly convex domains are necessarily connected}. In the present
paper we consider the three-phase partitioning of a domain (open and
connected subset) $\Omega\subset\mathbb{R}^{3}$ (or $\mathbb{R}^{2}$)
with boundary $\Sigma=\partial\Omega$, in which three phases coexist
by the formation of triple junctions. $\Sigma$ is assumed to be a
$C^{r}$-hypersurface of $\mathbb{R}^{3}$. Occasionally we present
definitions, formulas and propositions more generally in $\mathbb{R}^{n}$,
but our main results concern $\mathbb{R}^{3}$ and $\mathbb{R}^{2}$.

\begin{figure}[h]
\includegraphics[scale=0.75]{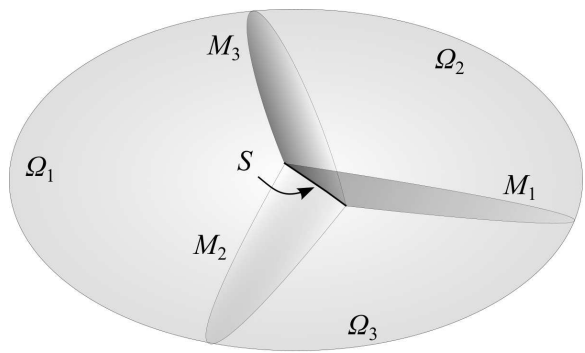}\caption{\label{fig:TJ}Single triple-junction partitioning. $\Omega_{i}$
is the space occupied by phase $i$ and $M_{i}$ is the interface
between phases $k$, $l$ with $k\protect\neq i$ and $l\protect\neq i$.
$S$ is the spine of the triple junction.}
\end{figure}

The problem can be mathematically formulated as follows. Let $\Omega\subset\mathbb{R}^{3}$
(or $\mathbb{R}^{2}$) be a domain with boundary $\Sigma=\partial\Omega$
as stated. In the volume constrained 3-phase partitioning problem
(refer to Figure \ref{fig:TJ}) we seek a division of $\Omega$ into
three subsets (the phases) $\Omega_{1}$, $\Omega_{2}$, $\Omega_{3}$,
each having prescribed volume $\left|\Omega_{i}\right|=V_{i}$, and
boundaries $M_{i}$ in $\Omega$ (the \emph{interfaces}), which form
a triple junction $T=(M_{i})_{i=1}^{3}$, such that the total interface
is a (local or global) minimizer of the area functional. The interfaces
$M_{i}$ are assumed to be $C^{2}$-hypersurfaces with boundary. In
a general setting, the interfaces $M_{1},\cdots,M_{m}$ form more
than one triple junction (see Figure \ref{fig:Intro-st-unst-TJs}).
The area, or more generally, the surface energy functional of the
partitioning is given by
\begin{equation}
A(T)=\sum_{i=1}^{m}\gamma_{i}A(M_{i})\label{eq:1-1}
\end{equation}
where $\gamma_{i}>0$ is the surface energy density (surface tension)
associated with $M_{i}$. In this notation, interfaces, phases and
subsets are identified by successive indexing. Other indexing schemes
are possible. The notation\cite{key-18} $M_{ij}$ for the interfaces,
where $i$, $j$ are phases in contact, was not found convenient for
our calculations. More convenient notations are introduced in the
sequel; refer to Example \ref{exa:primary} for a brief comparison
of indexing schemes. A \emph{minimal partitioning} is a critical point
of the functional (\ref{eq:1-1}),
\[
\delta A(T):=\left.\frac{d}{dt}A(T^{t})\right|_{t=0}=0,
\]
where $T^{t}$ is any admissible variation of the partition (for a
precise formulation of this, see Definition \ref{def:var}). The second
variation of the surface energy functional $A$ is defined by
\[
\delta^{2}A(T):=\left.\frac{d^{2}}{dt^{2}}A(T^{t})\right|_{t=0}
\]
A \emph{stable partition} is a minimal partition with $\delta^{2}A(T)>0$
for all nontrivial admissible variations. A partition is \emph{disconnected}
if at least one phase $\Omega_{i}$ is a disconnected subset (see
Figure \ref{fig:Intro-st-unst-TJs}).

Our main result is Theorem \ref{thm:stab-discon-tj} which establishes
the \emph{existence of stable partitions with a disconnected phase}
in convex domains of $\mathbb{R}^{2}$ by a configuration of two triple
junctions (see Figure \ref{fig:DTJ}):
\begin{thm*}[Existence of stable partitions with a disconnected phase in $\mathbb{R}^{2}$]
Let $\Omega$ be a convex domain in $\mathbb{R}^{2}$, and $T=(M_{1},\cdots,M_{5})$
a minimal disconnected three-phase partitioning of $\Omega$ by a
system of two $C^{2}$ triple junctions as in Figure \ref{fig:DTJ},
with volume constraints. Furthermore, for $\Omega$ and the partitioning
system $T$ we make the following assumptions:

\emph{(H1)} The boundary $\Sigma=\partial\Omega$ is $C^{2}$ in a
neighborhood of $\Sigma\cap\overline{T}$ and it is flat at $\overline{T}\cap\Sigma$.
In particular this means $\sigma=0$ at all points of $\overline{T}\cap\Sigma$.

\emph{(H2)} $M_{1}$ is flat, i.e. $\kappa_{1}=0$, and the length
of $M_{1}$ is $L$.

\emph{(H3) }All\emph{ }other leaves have the same curvature $\kappa\neq0$
and the same length $|M_{i}|=l$, $i=2,\cdots,5$.

\emph{(H4)} $\alpha<0$ in the orientation of Figure \ref{fig:DTJ}.

Then there is a $L_{0}>0$, possibly depending on $l$ and $\kappa$,
such that for $L\leqslant L_{0}$ the disconnected triple junction
partitioning $T$ is stable.
\end{thm*}
This is in contrast to the 2-phase partitioning, indicating that the
instability of disconnected partitions is specific to the 2-phase
partitioning. Unstable triple junction configurations of the same
topology in convex domains exist, as it is shown in Section \ref{sec:App-TJ-R3}.
Figure \ref{fig:Intro-st-unst-TJs} shows the geometric characteristics
of stable (type II) and unstable (type I) configurations. The quantity
$\alpha$ appears in the formula of the second variation of area for
a triple junction system (see equations \ref{eq:104}). We established
Theorem \ref{thm:stab-discon-tj} by proving that the second variation
of the area of the double triple junction system (which is by hypothesis
minimal, i.e. a critical point of the area functional) is positive
for all nontrivial admissible variations. The fundamentals of the
method are briefly presented in Section \ref{sec:Spectr-Anal} (see
also \cite{key-32}).

\begin{figure}
\includegraphics[scale=0.75]{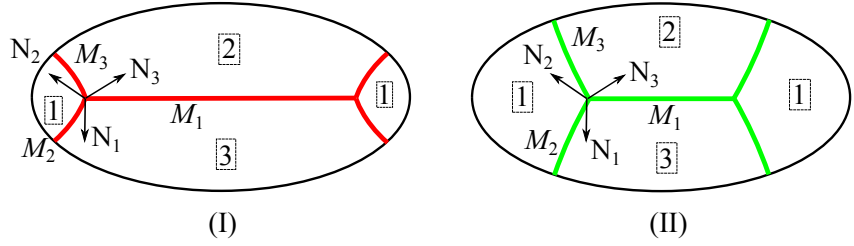}

\caption{\label{fig:Intro-st-unst-TJs}Disconnected triple junction configurations
in 2-dimensional space. Boxed numbers indicate phases. Phase 1 is
disconnected, while phases 2, 3 are connected in both cases. $M_{1},M_{2},M_{3}$
are the interfaces associated with the left triple junction. $N_{1},N_{2},N_{3}$
are the unit normal fields of the respective interfaces. In the shown
orientation, configuration (I) has $\alpha=\frac{\sqrt{3}}{4}\left(\kappa_{2}-\kappa_{3}\right)>0$,
while configuration (II) has $\alpha<0$. $\kappa_{2}$, $\kappa_{3}$
are the signed curvatures of the respective interfaces, defined by
$\kappa_{i}=\left<\frac{dT_{i}}{ds},N_{i}\right>$ for each interface,
$T_{i}$ being the unit tangent field of $M_{i}$, which is considered
parametrized by arc length $s$. }
\end{figure}

The basis of our analysis is a formula for the second variation of
area for minimal triple junction partitions with volume constraints
in $\mathbb{R}^{3}$ (Theorem \ref{thm:sec-var-TJ}):
\begin{thm*}[2nd variation of area for minimal triple junctions with volume constraints
in $\mathbb{R}^{3}$]
Let $\Omega$ be a domain in $\mathbb{R}^{3}$, $T=(T_{j})_{j=1}^{r}$,
$T_{j}=(M_{pj})_{p=1}^{3}$, a minimal three-phase partition of $\Omega$
by a set of $r$ $C^{2}$ triple junctions with volume constraints,
and $w$ an admissible variation satisfying (\ref{eq:3-4}) and the
volume constraints. On each leaf $M_{pj}$ we have the splitting $w=u_{pj}+v_{pj}$,
$u_{pj}\in TM_{pj}$, $v_{pj}\in NM_{pj}$, and we set $f_{pj}=w\cdot N_{pj}=v_{pj}\cdot N_{pj}$,
$N_{pj}$ being the unit normal field of $M_{pj}$. Then the following
formula holds for the second variation of the area functional,
\[
\begin{gathered}\delta^{2}A^{\star}(T)=\sum_{p,j}\gamma_{pj}\int_{M_{pj}}\left(|\mathrm{grad}_{M_{pj}}f_{pj}|^{2}-|B_{M_{pj}}|^{2}f_{pj}^{2}\right)\\
-\sum_{j=1}^{r}\sum_{p=1}^{3}\gamma_{pj}\int_{\partial M_{pj}\cap\varSigma}f_{pj}^{2}II_{\varSigma}(N_{pj},N_{pj})\\
+\sum_{j=1}^{r}\sum_{p=1}^{3}\gamma_{pj}\int_{S_{j}}f_{pj}h_{pj}II_{M_{pj}}(\nu_{pj},\nu_{pj})
\end{gathered}
\]
where $S_{j}$ is the spine of $T_{j}$ and $\nu_{pj}\in TM_{pj}$
is the unit normal field of $\partial M_{pj}\cap S_{j}$.
\end{thm*}
As an application of this theorem, we prove in Section \ref{sec:App-TJ-R3}
the stability of the previously mentioned class of triple junction
partitions.

In order to obtain Theorem \ref{thm:sec-var-TJ}, which holds in dimension
three, we first need to extend the second variation formula (\ref{eq:2-1})
in the following Proposition (Proposition \ref{prop:3-1.3}, Section
\ref{sec:2nd-var-formulas}), which works in all dimensions. The developed
formulas apply to constant mean curvature manifolds allowing for \emph{tangential
variations}:
\begin{prop*}[2nd variation of area for constant mean curvature manifolds in $\mathbb{R}^{3}$]
Let $M$ be $C^{2}$-hypersurface of $\mathbb{R}^{n}$ with boundary.
We assume that $M$ has constant mean curvature $\kappa$ and $w$
is a variation compactly supported in $M$ and whose support is contained
in a chart of $M$. Further let $N$ be the unit normal field of $M$
in a chart containing the support of $w$, $\nu$ the unit outward
normal of $\partial M$ which is tangent to $M$, $u=w^{\top}$, $v=w^{\bot}$,
and $f=w\cdot N$. Then the second variation of the area of $M$ is
given by
\[
\begin{alignedat}{1}\delta^{2}A(M)= & \int_{M}\left(|\mathrm{grad}_{M}f|^{2}-|B_{M}|^{2}f^{2}\right)+\\
 & \int_{M}\kappa\left[II_{M}(u,u)-2f\mathrm{div}_{M}u+\kappa f^{2}-a\cdot N\right]+\\
 & \int_{\partial M}\left[(u\cdot\nu)\mathrm{div}_{M}u-\left\langle \nabla_{u}u,\nu\right\rangle +2fII_{M}(u,\nu)+a\cdot\nu\right]
\end{alignedat}
\]
and for variations satisfying (\ref{eq:3-4})
\[
\begin{alignedat}{1}\delta^{2}A(M)= & \int_{M}\left(|\mathrm{grad}_{M}\,f\,|^{2}-|B_{M}|^{2}f^{2}\right)+\\
 & \int_{M}\kappa\left[II_{M}(u,u)-2f\mathrm{div}_{M}u+\kappa f^{2}-a\cdot N\right]+\\
 & \int_{\partial M}\left[(u\cdot\nu)\mathrm{div}_{M}u+fII_{M}(u,\nu)+f\left\langle D_{N}w,\nu\right\rangle \right]
\end{alignedat}
\]
where
\[
|B_{M}|^{2}=g^{ik}g^{jl}B_{ij}B_{kl}=B_{j}^{i}B_{i}^{j},\quad B_{ij}=II_{M}(E_{i},E_{j}),
\]
and $\left(E_{i}(p)\right)_{i=1}^{n-1}$ is a local basis of $T_{p}M$.
\end{prop*}
The precise formulation of the variational problems considered in
this paper, along with notation and well-known facts used, is given
in Section \ref{sec:Notation-and-Preliminaries}. For the reader's
convenience, in Section \ref{sec:1st-var} we briefly formulate facts
related to the first variation of area in a form suitable for triple
junction partitions.

\section{\label{sec:Notation-and-Preliminaries}Notation and Preliminaries}

Throughout this paper we use manifolds with boundary (see \cite{key-33}
p. 478 for a definition).
\begin{defn}
Let $\Omega$ be a domain of $\mathbb{R}^{3}$. By a $C^{r}$\emph{
triple junction} in $\Omega$ (see Figure \ref{fig:TJ}) we mean a
collection of three 2-dimensional $C^{r}$ submanifolds of $\mathbb{R}^{3}$
with boundary, $(M_{i})_{i=1}^{3}$, having the same boundary in $\Omega$
\[
\partial M_{i}\cap\Omega=S,\quad i=1,2,3
\]
which is called \emph{the spine} of the triple junction. We will refer
to the manifolds $M_{i}$ as \emph{the leaves} of the triple junction.
Triple junctions in $\mathbb{R}^{2}$ are defined analogously.
\end{defn}

\begin{defn}
\label{def:var}Let $M$ be a $n$-dimensional $C^{1}$ submanifold
of $\mathbb{R}^{m}$ with boundary, $V$ an open subset of $\mathbb{R}^{m}$
such that $V\cap M\neq\emptyset$. A variation of $M$ is a collection
of diffeomorphisms $(\xi^{t})_{t\in I}$, $I=\,]-\delta,\delta[$,
$\delta>0$, $\xi^{t}:V\to V$ such that\cite{key-5}

(i) The function $\xi(t,x)=\xi^{t}(x)$ is $C^{2}$

(ii) $\xi^{0}=id_{V}$

(iii) $\xi^{t}|_{V\setminus K}=id_{V\setminus K}$ for some compact
set $K\subset V$. \hfill{}$\square$
\end{defn}

In place of the $\xi^{t}$ we often consider their extension by identity
to all of $\mathbb{R}^{m}$. With each variation we associate the
\emph{first} and \emph{second variation fields}
\[
w(x)=\xi_{t}(0,x),\quad a(x)=\xi_{tt}(0,x)
\]
also known as \emph{velocity} and \emph{acceleration fields}\cite{key-6},
$\xi_{t}$, $\xi_{tt}$ denoting first and second partial derivatives
in $t$. By the \emph{support of a variation} we mean the support
of $w$. We set $M^{t}:=\xi^{t}(M)$ for the variation of $M$.

This definition extends readily to triple junctions, which, as individual
geometric objects, are not submanifolds of $\mathbb{R}^{m}$. Letting
$T=(M_{i})_{i=1}^{3}$ be a triple junction, and $(\xi^{t})_{t\in I}$
a variation, the triple junctions 
\[
T^{t}=(\xi^{t}(M_{i}))_{i=1}^{3},\;t\in I
\]
are a variation of $T$. The 1-dimensional submanifolds
\[
S^{t}=\xi^{t}(S),\;t\in I
\]
are a variation of the spine of $T$.

Let $M$ be a submanifold of $\mathbb{R}^{m}$. For a vector field
$X$ of $\mathbb{R}^{m}$ defined on a domain $V$ of $\mathbb{R}^{m}$
we define its tangent and normal parts $X^{\top}(p)\in T_{p}M$ and
$X^{\bot}(p)\in N_{p}M$, $p\in M\cap V$. $TM$ and $NM$ are the
tangent and normal bundles of $M$ respectively. The notation $X\in TM$
is an abbreviation for $X(p)\in T_{p}M$ for $p$ in an open subset
of $M$. Given any two vector fields $X$, $Y$ on $V$, $D_{Y}X$
denotes the directional derivative of $X$ in direction $Y$ in $\mathbb{R}^{m}$.
When $u,v\in TM$, 
\[
\nabla_{v}u=(D_{v}u)^{\top}\in TM
\]
is the \emph{covariant derivative of $u$ in direction} $v$. The
\emph{covariant derivative} $\nabla u$ of $u$ is a $(1,1)$-tensor
field defined by
\[
\nabla u(\omega,v)=\omega(\nabla_{v}u),\quad v\in TM,\:\omega\in T^{\star}M,
\]
where $T_{p}^{\star}M$ is the dual space of $T_{p}M$. The \emph{components
of the covariant derivative} $\nabla u$ are defined by
\[
u_{\,|j}^{i}=\nabla u(E^{i},E_{j})=dq^{i}\left(\nabla_{E_{j}}u\right).
\]
$E_{i}(p)$ is a basis of $T_{p}M$ and $E^{j}(p)\equiv dq^{j}(p)$
is the corresponding dual basis of $T_{p}^{\star}M$.

The normal part of the directional derivative $B(u,v)=(D_{v}u)^{\bot}\in NM$
defines the \emph{2nd fundamental form} tensor. The \emph{mean curvature
vector} is given by the trace of this tensor
\begin{equation}
H(p)=\sum_{i}B_{p}(E_{i},E_{i})\label{eq:def-H}
\end{equation}
in an \emph{orthonormal basis} $(E_{i})$ of $T_{p}M$. If $M$ is
a hypersurface, the \emph{scalar mean curvature} is defined by
\begin{equation}
\kappa=H\cdot N\label{eq:def-mean-curv}
\end{equation}
where $N$ is a unit normal field of $M$. Similarly, when $M$ is
a hypersurface of $\mathbb{R}^{m}$, we define the \emph{scalar version
of the 2nd fundamental form}:
\begin{equation}
II_{M}(u,v)=B(u,v)\cdot N\label{eq:def-mean-2nd-ff}
\end{equation}
The \emph{Weingarten mapping} of a hypersurface $M$ of $\mathbb{R}^{n+1}$
is defined fiberwise by
\[
W(p):T_{p}M\to T_{p}M,\;u\mapsto D_{u}N\quad(p\in M)
\]

When index notation is used, \emph{summation over pairs of identical
indices} (which for tensor expressions must be pairs of contravariant-convariant
indices) is assumed throughout.

The \emph{gradient} of a function $f:U\to\mathbb{R}$ defined in an
open neighborhood $U$ of $M$, is given by
\[
\mathrm{grad}_{M}f=g^{ij}\frac{\partial f}{\partial q^{j}}E_{i}
\]
in a coordinate system $q^{1},\cdots,q^{n}$. In this definition $g^{ij}$
are the contravariant components of the metric tensor $g_{ij}=E_{i}\cdot E_{j}$.
The \emph{divergence} of a tangent vector field $v$ of $M$ defined
in $U$ is the trace of its covariant derivative, i.e.
\[
\mathrm{div}v=\mathrm{tr}\nabla u=\nabla u(E_{i},E^{i})=u_{\,|i}^{i}
\]
For a general vector field $w$ which is not tangent to $M$ the divergence
is defined by
\[
\mathrm{div}_{M}w=\left\langle D_{E_{i}}w,E^{i}\right\rangle =g^{ij}\left\langle D_{E_{i}}w,E_{j}\right\rangle 
\]
 The notation $\left\langle \cdot,\cdot\right\rangle $ is alternately
used to denote scalar product in lengthier expressions.

For the first variation of the area of a manifold with boundary, the
following proposition holds.
\begin{prop}
Let $M$ be a $n$-dimensional $C^{1}$-submanifold of $\mathbb{R}^{m}$
with boundary, and $\xi^{t}$ a variation as in Definition \ref{def:var},
which is compactly supported in $M$. We assume that the support of
$\xi^{t}$ is contained in a chart of $M$. Then the first variation
of the area functional $A$ is given by
\begin{equation}
\delta A(M)=\left.\frac{d}{dt}A(M^{t})\right|_{t=0}=\int_{M}\mathrm{div}{}_{M}wdS\label{eq:1-9}
\end{equation}
If $M$ is additionally a $C^{2}$-hypersurface with boundary,
\begin{equation}
\delta A(M)=-\int_{M}H\cdot wdS+\int_{\partial M}w\cdot\nu ds\label{eq:1-21}
\end{equation}
where $\nu$ is the unit outward pointing normal vector field of $\partial M$
which is tangent to $M$, and $H$ is the mean curvature vector.
\end{prop}

For the proof see \cite{key-6,key-20}. For brevity we will omit the
integration symbols $dS$, $ds$.

Our formulas for the second variation of the area of triple junctions
derive from the following well-known result.
\begin{prop}
Let $M$ be a $n$-dimensional $C^{2}$-submanifold of $\mathbb{R}^{m}$
with boundary, and $\xi^{t}$ a variation of $M$ compactly supported
in $M$, and the support of $\xi^{t}$ is contained in a chart of
$M$. Then the second variation of the area functional is given by
\begin{equation}
\begin{alignedat}{1}\delta^{2}A(M) & =\int_{M}\left(\mathrm{div}_{M}a+(\mathrm{div}_{M}w)^{2}+g^{ij}\left\langle (D_{E_{i}}w)^{\bot},(D_{E_{j}}w)^{\bot}\right\rangle \right.\\
 & \left.-g^{ik}g^{jl}\left\langle (D_{E_{i}}w)^{\top},E_{j}\right\rangle \left\langle (D_{E_{l}}w)^{\top},E_{k}\right\rangle \right).
\end{alignedat}
\label{eq:2-1}
\end{equation}
$(E_{i})_{i=1}^{n}$ are the basis vector fields in a chart containing
the support of the variation.
\end{prop}

For the proof see \cite{key-6}. 

The above formulas for the first variation of area extend readily
to triple junctions:
\begin{equation}
\delta A(T)=\left.\frac{d}{dt}A(T^{t})\right|_{t=0}=\sum_{i=1}^{3}\gamma_{i}\int_{M_{i}}\mathrm{div}{}_{M_{i}}w\label{eq:1-23}
\end{equation}
\begin{equation}
\delta A(T)=-\sum_{i=1}^{3}\gamma_{i}\int_{M_{i}}H\cdot w+\int_{S}w\cdot\sum_{i=1}^{3}\gamma_{i}\nu_{i}+\sum_{i=1}^{3}\gamma_{i}\int_{\partial M_{i}\cap\Sigma}w\cdot\nu_{i}\label{eq:1-24}
\end{equation}
The extension of formula (\ref{eq:2-1}) for the second variation
of area to triple junctions is not as straight-forward. This task
is undertaken in Section \ref{sec:2nd-var-formulas} after proper
modifications of (\ref{eq:2-1}).

\section{\label{sec:1st-var}First Variation-Young's Equality}

We treat the constraints of triple junction partitions by using Lagrange
multipliers. In the simplest case of a connected partitioning, in
which one triple junction is present, we consider the following modified
area functional
\begin{equation}
A^{\star}(T)=\sum_{i=1}^{3}\gamma_{i}A(M_{i})-\sum_{j=1}^{2}\lambda_{j}(|\Omega_{j}|-V_{j})\label{eq:1-26}
\end{equation}
where $|\Omega_{j}|$ is the Lebesque measure of $\Omega_{j}$, and
$V_{j}$ is the prescribed value for the volume of $\Omega_{j}$.
The introduction of Lagrange multipliers is a matter of convenience,
and one could proceed without them by properly restricting admissible
variations to those preserving the volumes of the $\Omega_{j}$ (see
\cite{key-7,key-32}). In taking the variations of (\ref{eq:1-26})
we can drop the constants $V_{i}$ altogether.

The leaves of a minimal triple junction with volume constraint are
at angles $\vartheta_{i}$ according to Young's law. We formulate
this well-known fact in the simple case of a single triple junction
and then extend it to a more general setting.
\begin{prop}
\label{prop:Young-constr}Let $T=(M_{i})_{i=1}^{3}$ be a $C^{2}$
triple junction partition of $\Omega\subset\mathbb{R}^{3}$ into the
subdomains $\Omega_{j}$ ($j=1,2,3$). If $T$ is minimal, then Young's
equality
\begin{equation}
\frac{\sin\vartheta_{1}}{\gamma_{1}}=\frac{\sin\vartheta_{2}}{\gamma_{2}}=\frac{\sin\vartheta_{3}}{\gamma_{3}}\label{eq:1-30}
\end{equation}
holds on the spine $S$ of $T$. Furthermore, the leaves have constant
(scalar) mean curvature satisfying the relation
\begin{equation}
\gamma_{1}\kappa_{1}+\gamma_{2}\kappa_{2}+\gamma_{3}\kappa_{3}=0\label{eq:1-31}
\end{equation}
where $\kappa_{i}=H_{i}\cdot N_{i}$ is the mean curvature of $M_{i}$. 
\end{prop}

\begin{proof}
Let $(\xi^{t})_{t\in I}$ be any variation with first variation field
$w$. By (\ref{eq:1-26}) we obtain
\[
\delta A^{\star}(T)=\delta A(T)-\sum_{i=1}^{2}\lambda_{i}\left.\frac{d}{dt}|\Omega_{i}^{t}|\right|_{t=0}
\]
where $\Omega_{i}^{t}=\xi^{t}(\Omega_{i})$. Using (\ref{eq:1-24})
and
\begin{equation}
\left.\frac{d}{dt}|\Omega_{i}^{t}|\right|_{t=0}=\int_{\partial\Omega_{i}}w\cdot N_{\partial\Omega_{i}}\label{eq:var-omega}
\end{equation}
where $N_{\partial\Omega_{i}}$ is the unit \emph{outward} normal
field of $\partial\Omega_{i}$, we obtain
\begin{equation}
\delta A^{\star}(T)=\int_{S}w\cdot\sum_{i=1}^{3}\gamma_{i}\nu_{i}-\sum_{i=1}^{3}\gamma_{i}\int_{M_{i}}H\cdot w-\sum_{i=1}^{2}\lambda_{i}\sum_{j\neq i}\int_{M_{j}}N_{\partial\Omega_{i}}\cdot w\label{eq:A-4}
\end{equation}

Expanding out the last two terms on the right side of this equality
and collecting integrals on the same manifold, we obtain
\begin{equation}
\begin{gathered}\int_{M_{1}}(\gamma_{1}H_{1}-\lambda_{2}N_{1})\cdot w+\int_{M_{2}}(\gamma_{2}H_{2}+\lambda_{1}N_{2})\cdot w\\
+\int_{M_{3}}(\gamma_{3}H_{3}-\lambda_{1}N_{3}+\lambda_{2}N_{3})\cdot w
\end{gathered}
\label{eq:A-5}
\end{equation}
Considering successively variations concentrated in the interior of
$M_{1}$, $M_{2}$, $M_{3}$ we obtain
\[
\kappa_{1}\gamma_{1}-\lambda_{2}=0,\quad\kappa_{2}\gamma_{2}+\lambda_{1}=0,\quad\kappa_{3}\gamma_{3}-\lambda_{1}+\lambda_{2}=0
\]
Addition of these three equations gives (\ref{eq:1-31}). Furthermore,
all integrals in (\ref{eq:A-5}) cancel out and (\ref{eq:A-4}) reduces
to
\[
\delta A^{\star}(T)\cdot w=\int_{S}w\cdot\sum_{i=1}^{3}\gamma_{i}\nu_{i}=0
\]
hence
\[
\sum_{i=1}^{3}\gamma_{i}\nu_{i}=0
\]
Recalling that $\nu_{i}(p)\in T_{p}M_{i}$ and observing that the
vectors $\gamma_{i}\nu_{i}$ ($i=1,2,3$) form a triangle, by the
sine law of Euclidean geometry we obtain (\ref{eq:1-30}).
\end{proof}
The presence of many triple junctions requires consideration of the
following modified functional:
\begin{equation}
A^{\star}(M)=A(M)-\sum_{j=1}^{2}\lambda_{j}\left(\sum_{k=1}^{P_{j}}\left|\Omega_{jk}\right|-V_{j}\right)\label{eq:modif-area-func}
\end{equation}
In this formula, $P_{j}$ is the number of distinct sets which comprise
phase $j$ (indexed by $k$); $V_{j}$ is the volume of phase $j$,
and $\lambda_{j}$ is the Lagrange multiplier corresponding to the
volume constraint for the $j$-th phase. Since $\sum_{j=1}^{3}\sum_{k=1}^{P_{j}}\left|\Omega_{jk}\right|=\left|\Omega\right|$,
there are only two linearly independent constraints.

\begin{figure}
\includegraphics[scale=0.7]{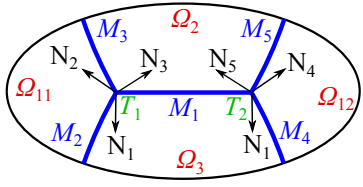}

\caption{\label{fig:discon-d-tj}Disconnected triple junction partition for
Example \ref{exa:primary} and Proposition \ref{prop:1st-var-gen}.}

\end{figure}
\begin{example}
\label{exa:primary}For the disconnected 3-phase partitioning of $\Omega$
(see Figure \ref{fig:discon-d-tj}) by a system of two triple junctions,
the modified area functional is given by
\begin{equation}
A^{\star}(M)=A(T)-\lambda_{2}\left|\Omega_{2}\right|-\lambda_{3}\left|\Omega_{3}\right|\label{eq:example-func}
\end{equation}
In this expression, 
\[
A(T)=\sum_{i=1}^{5}\gamma_{i}A(M_{i}),
\]
$\gamma_{1}=\epsilon_{23}$, $\gamma_{2}=\gamma_{4}=\epsilon_{13}$,
$\gamma_{3}=\gamma_{5}=\epsilon_{12}$, and $\epsilon_{ij}=\epsilon_{ji}$
is the interfacial energy density of the interface separating phases
$i$, $j$. The volume constants $V_{j}$ were dropped as they play
no part in the variational process. On using the volume constraints
for phases 1 and 2, the modified area functional assumes the form
\begin{equation}
A^{\star}(M)=A(T)-\lambda_{1}\left(\left|\Omega_{11}\right|+\left|\Omega_{12}\right|\right)-\lambda_{2}\left|\Omega_{2}\right|\label{eq:example-func-1}
\end{equation}
We can also use all three constraints, which does not alter the final
formulas and results. Occasionally, we use a notation indicating the
triple junction, which an interface belongs to. In this notation the
area functional becomes
\[
A(T)=\sum_{p=1}^{3}\sum_{j=1}^{2}\gamma_{pj}A(M_{pj})
\]
The indices $p$ and $j$ stand for ``phase'' and ``junction'',
$M_{pj}$ is the interface opposite to phase $p$ at junction $j$
(for example, in Figure \ref{fig:discon-d-tj} $M_{22}$ is shown
as $M_{4}$) and $\gamma_{pj}$ the corresponding surface tension.
Since there is only one interface between phases 2 and 3, $M_{11}\equiv M_{12}\equiv M_{1}$,
and this term occurs only once in the sum. In a similar fashion, we
drop superfluous indices from subsets. For example, referring to Figure
\ref{fig:discon-d-tj}, we write $\Omega_{2}$ and $\Omega_{3}$ instead
of $\Omega_{21}$ and $\Omega_{31}$. The connection between notations
is
\begin{equation}
\begin{array}{cccccccc}
M_{11} & \gamma_{11} & \leftrightarrow & M_{1} & \gamma_{1} & \leftrightarrow & M_{23}^{(1)}\equiv M_{23}^{(2)} & \epsilon_{23}\\
M_{21} & \gamma_{21} & \leftrightarrow & M_{2} & \gamma_{2} & \leftrightarrow & M_{13}^{(1)} & \epsilon_{13}\\
M_{31} & \gamma_{31} & \leftrightarrow & M_{3} & \gamma_{3} & \leftrightarrow & M_{12}^{(1)} & \epsilon_{12}\\
M_{22} & \gamma_{22} & \leftrightarrow & M_{4} & \gamma_{4} & \leftrightarrow & M_{13}^{(2)} & \epsilon_{13}\\
M_{32} & \gamma_{32} & \leftrightarrow & M_{5} & \gamma_{5} & \leftrightarrow & M_{12}^{(2)} & \epsilon_{12}
\end{array}\label{eq:conn-not}
\end{equation}
The first two columns refer to the second notation, the next two to
the first notation and the last two to the interfacial system (see
Section \ref{sec:Intro}, text below equation (\ref{eq:1-1})). In
the latter, $M_{pq}^{(j)}$ is the interface between phases $p,q$
attached to triple junction $j$. \hfill{}$\square$

The following theorem extends Proposition \ref{prop:Young-constr}
to general disconnected triple junction partitions.
\end{example}

\begin{prop}
\label{prop:1st-var-gen}Let $T=(T_{j})_{j=1}^{r}=(M_{pj})_{p=1,\cdots,3;j=1,\cdots,r}$
be a three-phase partitioning of $\Omega\subset\mathbb{R}^{3}$ by
a system of $r$ $C^{2}$-triple junctions, into the domains $\Omega_{pj}$
$(j=1,\cdots,r;p=1,2,3)$. Further let $N_{pj}$ be the unit normal
field of $M_{pj}$ and \textup{$N_{\Sigma}$} be the unit normal field
of $\Sigma$. If $T$ is minimal, then

\emph{(i)} Young's equality (\ref{eq:1-30}) holds for each triple
junction in the system. Equivalently, the following equalities
\begin{equation}
\sum_{p=1}^{3}\gamma_{pj}\nu_{pj}=0\label{eq:3.10-1}
\end{equation}
hold for all triple junctions $j$ of the system.

\emph{(ii)} The scalar mean curvature $\kappa_{pj}=H_{pj}\cdot N_{pj}$
of each interface $M_{pj}$ is constant and
\begin{equation}
\kappa_{p1}=\cdots=\kappa_{pr}\equiv\kappa_{p}\label{eq:id-mean-curv}
\end{equation}
for all $p=1,2,3$.

\emph{(iii)} The scalar mean curvatures $\kappa_{p}$ satisfy the
relation
\begin{equation}
\sum_{p=1}^{3}\gamma_{p}\kappa_{p}=0\label{eq:mean-curv-id}
\end{equation}
on each triple junction.

\emph{(iv)} Each $M_{pj}$ is normal to $\Sigma$, i.e. on each $M_{pj}\cap\Sigma$
we have $N_{pj}\cdot N_{\Sigma}=0$ or $N_{pj}\in T\Sigma$.
\end{prop}

\begin{rem}
\label{rem:1st-var-gen}Equality (\ref{eq:mean-curv-id}) holds for
each triple junction,
\[
\sum_{p=1}^{3}\gamma_{pj}\kappa_{pj}=0
\]
However, in view of (\ref{eq:id-mean-curv}) and the fact that $\gamma_{pj}=\gamma_{p}$
for all $j$ (see correspondence table (\ref{eq:conn-not})\-), all
these equalities reduce to the single equality (\ref{eq:mean-curv-id}).
\end{rem}

\begin{proof}
For concreteness we consider the disconnected 3-phase partitioning
of Fig. \ref{fig:discon-d-tj} with the indicated orientation. Letting
$w$ be any variation of $T$, by (\ref{eq:example-func}) in view
of (\ref{eq:1-21}) and (\ref{eq:var-omega}) we obtain
\[
\begin{alignedat}{1}\delta A^{\star}(T) & =\sum_{j=1}^{2}\sum_{p=1}^{3}\gamma_{pj}\delta A(M_{pj})-\lambda_{2}\delta\left|\Omega_{2}\right|-\lambda_{3}\delta\left|\Omega_{3}\right|\\
 & =-\sum_{j=1}^{2}\sum_{p=1}^{3}\gamma_{pj}\int_{M_{pj}}H_{pj}\cdot w\\
 & +\sum_{j=1}^{2}\sum_{p=1}^{3}\gamma_{pj}\int_{S}\nu_{pj}\cdot w+\sum_{j=1}^{2}\sum_{p=1}^{3}\gamma_{pj}\int_{\partial M_{pj}\cap\Sigma}\nu_{pj}\cdot w\\
 & -\lambda_{2}\int_{M_{1}}N_{1}\cdot w-\lambda_{2}\int_{M_{31}}(-N_{31})\cdot w-\lambda_{2}\int_{M_{32}}(-N_{32})\cdot w\\
 & -\lambda_{3}\int_{M_{1}}(-N_{1})\cdot w-\lambda_{3}\int_{M_{21}}N_{21}\cdot w-\lambda_{3}\int_{M_{22}}N_{22}\cdot w
\end{alignedat}
\]
Rearranging gives
\begin{equation}
\begin{alignedat}{1}\delta A^{\star}(T) & =\int_{M_{1}}\left[\left(\lambda_{3}-\lambda_{2}\right)N_{1}-\gamma_{1}H_{1}\right]\cdot w\\
 & +\int_{M_{21}}\left(-\lambda_{3}N_{21}-\gamma_{21}H_{21}\right)\cdot w+\int_{M_{31}}\left(\lambda_{2}N_{31}-\gamma_{31}H_{31}\right)\cdot w\\
 & +\int_{M_{22}}\left(-\lambda_{3}N_{22}-\gamma_{22}H_{22}\right)\cdot w+\int_{M_{32}}\left(\lambda_{2}N_{32}-\gamma_{32}H_{32}\right)\cdot w\\
 & +\int_{S}\left(\sum_{j=1}^{2}\sum_{p=1}^{3}\gamma_{pj}\nu_{pj}\right)\cdot w+\sum_{j=1}^{2}\sum_{p=1}^{3}\gamma_{pj}\int_{\partial M_{pj}\cap\Sigma}\nu_{pj}\cdot w
\end{alignedat}
\label{eq:3.13}
\end{equation}
Using variations concentrated on each leaf gives
\begin{equation}
\begin{array}{l}
-\gamma_{1}H_{1}\cdot N_{1}+\lambda_{3}-\lambda_{2}=0\\
-\gamma_{21}H_{21}\cdot N_{21}-\lambda_{3}=0\\
-\gamma_{31}H_{31}\cdot N_{31}+\lambda_{2}=0\\
-\gamma_{22}H_{22}\cdot N_{22}-\lambda_{3}=0\\
-\gamma_{32}H_{32}\cdot N_{32}+\lambda_{2}=0
\end{array}\label{eq:3.14}
\end{equation}
By (\ref{eq:3.13}), using variations concentrated on each spine,
we obtain
\begin{equation}
\sum_{p=1}^{3}\gamma_{pj}\nu_{pj}=0,\quad j=1,2\label{eq:3.15}
\end{equation}
Variations concentrated on each $\partial M_{pj}\cap\Sigma$ give
\begin{equation}
\nu_{pj}\cdot w=0\label{eq:3.15-1}
\end{equation}
for $p,j$ such that $\partial M_{pj}\cap\Sigma\neq\emptyset$. From
these relations it follows without difficulty that $\nu_{pj}\in N\Sigma$
and this proves (iv).

Part (i) follows from equations (\ref{eq:3.15}) by the same argumentation
applied in the proof of Proposition \ref{prop:Young-constr}. By the
second and fourth of (\ref{eq:3.14}), on account of $\gamma_{21}=\gamma_{22}$
we obtain
\[
\kappa_{21}=\kappa_{22}
\]
and in a similar fashion from the third and fifth of (\ref{eq:3.14})
\[
\kappa_{31}=\kappa_{32}
\]
The equality $\kappa_{11}=\kappa_{12}$ is trivial, and this proves
(ii). The constancy of the $\kappa$'s follows immediately from (\ref{eq:3.14}).
Addition of the first three equations of (\ref{eq:3.14}) gives $\sum_{p=1}^{3}\gamma_{p1}\kappa_{p1}=0$
and addition of the first, fourth and fifth gives $\sum_{p=1}^{3}\gamma_{p2}\kappa_{p2}=0$.
By Remark \ref{rem:1st-var-gen} these are identical and this proves
(iii).
\end{proof}

\section{\label{sec:2nd-var-formulas}Second Variation Formulas}

Formula (\ref{eq:2-1}) for the second variation of area is quite
general but not directly applicable. Here we derive a more convenient
expression for hypersurfaces with constant mean curvature, which in
a subsequent step is applied to multijunction partitions. To satisfy
the condition of fixed container walls, following \cite{key-7}, we
define admissible variations by the solutions of the initial value
problem
\begin{equation}
\frac{d\xi}{dt}=w(\xi),\quad\xi(0)=x\label{eq:3-4}
\end{equation}
Letting $\xi_{x}$ be the solution of (\ref{eq:3-4}) for the initial
condition $\xi_{x}(0)=x$, we set $\xi(x,t)=\xi_{x}(t)$ for the corresponding
variation. For solid undeformable walls we choose $w$ so that $w(p)\in T_{p}\Sigma$
for any $p\in\Sigma=\partial\Omega$.

On taking the time derivative of (\ref{eq:3-4}) we obtain the following
expression for the second variation field:
\begin{equation}
a=D_{w}w\label{eq:A-13}
\end{equation}
\begin{prop}
\label{prop:3-1.3}Let $M$ be $C^{2}$-hypersurface of $\mathbb{R}^{n}$
with boundary. We assume that $M$ has constant mean curvature $\kappa$
and $w$ is a variation compactly supported in $M$ and whose support
is contained in a chart of $M$. Further let $N$ be the unit normal
field of $M$ in a chart containing the support of $w$, $\nu$ the
unit outward normal of $\partial M$ which is tangent to $M$, $u=w^{\top}$,
$v=w^{\bot}$, and $f=w\cdot N$. Then the second variation of the
area of $M$ is given by
\begin{equation}
\begin{alignedat}{1}\delta^{2}A(M)= & \int_{M}\left(|\mathrm{grad}_{M}f|^{2}-|B_{M}|^{2}f^{2}\right)+\\
 & \int_{M}\kappa\left[II_{M}(u,u)-2f\mathrm{div}_{M}u+\kappa f^{2}-a\cdot N\right]+\\
 & \int_{\partial M}\left[(u\cdot\nu)\mathrm{div}_{M}u-\left\langle \nabla_{u}u,\nu\right\rangle +2fII_{M}(u,\nu)+a\cdot\nu\right]
\end{alignedat}
\label{eq:3-19}
\end{equation}
and for variations satisfying (\ref{eq:3-4})
\begin{equation}
\begin{alignedat}{1}\delta^{2}A(M)= & \int_{M}\left(|\mathrm{grad}_{M}\,f\,|^{2}-|B_{M}|^{2}f^{2}\right)+\\
 & \int_{M}\kappa\left[II_{M}(u,u)-2f\mathrm{div}_{M}u+\kappa f^{2}-a\cdot N\right]+\\
 & \int_{\partial M}\left[(u\cdot\nu)\mathrm{div}_{M}u+fII_{M}(u,\nu)+f\left\langle D_{N}w,\nu\right\rangle \right]
\end{alignedat}
\label{eq:3-19-1}
\end{equation}
where
\[
|B_{M}|^{2}=g^{ik}g^{jl}B_{ij}B_{kl}=B_{j}^{i}B_{i}^{j},\quad B_{ij}=II_{M}(E_{i},E_{j}),
\]
and $\left(E_{i}(p)\right)_{i=1}^{n-1}$ is a local basis of $T_{p}M$.
\end{prop}

\begin{proof}
To perform the calculations in an orderly fashion we set
\[
I:=g^{ij}\left\langle (D_{E_{i}}w)^{\bot},(D_{E_{j}}w)^{\bot}\right\rangle 
\]
\[
J:=g^{ik}g^{jl}\left\langle (D_{E_{i}}w)^{\top},E_{j}\right\rangle \left\langle (D_{E_{l}}w)^{\top},E_{k}\right\rangle 
\]
\[
K:=\mathrm{div}_{M}a,\quad L:=(\mathrm{div}_{M}w)^{2}.
\]
Then, noticing that all these items are bilinear in $w$ (for $K$
this will become clear shortly), we name the terms resulting from
breaking $w$ into tangent and normal parts, by adding the indices
$n$ and $t$ denoting normal and tangent parts respectively. For
example, for $I$ we have $I=I_{tt}+I_{tn}+I_{nt}+I_{tt}$, where
$I_{tt}=g^{ij}\left\langle (D_{E_{i}}u)^{\bot},(D_{E_{j}}u)^{\bot}\right\rangle $,
$I_{tn}=g^{ij}\left\langle (D_{E_{i}}u)^{\bot},(D_{E_{j}}v)^{\bot}\right\rangle $
etc.

Using this notation, the generic formula for the second variation
of area (\ref{eq:2-1}) reads
\begin{equation}
\begin{alignedat}{1}\delta^{2}A(M) & =\int_{M}(K+L+I_{tt}+2I_{tn}+I_{nn}-J_{tt}-2J_{tn}-J_{nn})\\
 & =\int_{M}\left[(I_{nn}-J_{nn})+2(I_{tn}-J_{tn})+(L+I_{tt}-J_{tt})+K\right].
\end{alignedat}
\label{eq:A-14}
\end{equation}
Since calculations are often done more conveniently in coordinate
notation, while final results are more concisely expressed coordinate-free,
we give all items in both notations. We are using summation convention
throughout. From
\begin{equation}
D_{E_{i}}v=D_{E_{i}}(fN)=fD_{E_{i}}N+(D_{E_{i}}f)N\label{eq:A-15}
\end{equation}
recalling that $D_{E_{i}}N\in TM$ we obtain
\begin{equation}
(D_{E_{i}}v)^{\bot}=(D_{E_{i}}f)N=\frac{\partial f}{\partial q^{i}}N\label{eq:3-25}
\end{equation}
where $q^{1},q^{2},\cdots,q^{n-1}$ is a local coordinate system of
$M$. We are using the same notation for $f:M\to\mathbb{R}$ and $f\circ\mathbf{x}$,
$\mathbf{x}$ being a parametrization (inverse chart mapping) of $M$.
By (\ref{eq:3-25})
\begin{equation}
I_{nn}:=g^{ij}\left\langle (D_{E_{i}}v)^{\bot},(D_{E_{j}}v)^{\bot}\right\rangle =g^{ij}\frac{\partial f}{\partial q^{i}}\frac{\partial f}{\partial q^{j}}=\left|\mathrm{grad}_{M}f\right|^{2}\label{eq:3-25-1}
\end{equation}
Since $u\cdot N=0$ and $\dim N_{p}M=1$ $(p\in M)$,
\[
\begin{alignedat}{1}I_{tt} & =g^{ij}\left\langle (D_{E_{i}}u)^{\bot},(D_{E_{j}}u)^{\bot}\right\rangle =g^{ij}\left\langle (D_{E_{i}}u)^{\bot},N\right\rangle \left\langle (D_{E_{j}}u)^{\bot},N\right\rangle \\
 & =g^{ij}\left\langle D_{E_{i}}u,N\right\rangle \left\langle D_{E_{j}}u,N\right\rangle =g^{ij}\left\langle u,D_{E_{i}}N\right\rangle \left\langle u,D_{E_{j}}N\right\rangle 
\end{alignedat}
\]
By the definition of the Weingarten mapping,
\begin{equation}
I_{tt}=|Wu|^{2}=B_{ki}B_{\,j}^{k}u^{i}u^{j}.\label{eq:3-26}
\end{equation}
\begin{sloppypar}Here $u^{1},\cdots,u^{n-1}$ are the components
of $u$ in the coordinate system $q^{1},\cdots,q^{n-1}$, $B_{ij}=II_{M}(E_{i},E_{j})$
the components of the 2nd fundamental form and $B_{\,j}^{i}=g^{ik}B_{kj}$
the corresponding mixed contravariant-covariant components. In the
derivation of (\ref{eq:3-26}) the following identity was used\end{sloppypar}
\begin{equation}
D_{E_{i}}N=WE_{i}=-B_{\,i}^{k}E_{k}.\label{eq:3-21}
\end{equation}
By (\ref{eq:3-25}) and the definition of the Weingarten mapping we
obtain
\[
\begin{alignedat}{1}I_{tn} & =g^{ij}\left\langle (D_{E_{i}}u)^{\bot},(D_{E_{j}}v)^{\bot}\right\rangle =g^{ij}\left\langle D_{E_{i}}u,(D_{E_{j}}v)^{\bot}\right\rangle \\
 & =g^{ij}\frac{\partial f}{\partial q^{j}}\left\langle D_{E_{i}}u,N\right\rangle =-g^{ij}\frac{\partial f}{\partial q^{j}}\left\langle u,D_{E_{i}}N\right\rangle \\
 & =-\left\langle u,D_{\mathrm{grad}_{M}f}N\right\rangle =-\left\langle u,W\mathrm{grad}_{M}f\right\rangle 
\end{alignedat}
\]
By the self-adjointness of the Weingarten mapping and (\ref{eq:3-21})
we obtain the following alternative expressions for $I_{tn}$:
\begin{equation}
I_{tn}=-\left\langle Wu,\mathrm{grad}_{M}f\right\rangle =II_{M}(u,\mathrm{grad}_{M}f)=B_{\,i}^{j}u^{i}\frac{\partial f}{\partial q^{j}}.\label{eq:3-28}
\end{equation}
It is easily checked that $I_{tn}=I_{nt}$.

We proceed to the calculation of $J$-terms. Recalling the definition
of covariant derivative from Section \ref{sec:Notation-and-Preliminaries}
and its component notation $\nabla_{Y}X=X_{|k}^{i}Y^{k}E_{i}$, $X,Y\in TM$,
we have 
\[
J_{tt}=g^{ik}g^{jl}\left\langle (D_{E_{i}}u)^{\top},E_{j}\right\rangle \left\langle (D_{E_{l}}u)^{\top},E_{k}\right\rangle =u_{\,|i}^{k}u_{\,|k}^{i}
\]
On using the following notation for the double contraction of two
tensors $S,T$
\[
S:T=S_{\,j}^{i}T_{\,i}^{j}
\]
we obtain
\begin{equation}
J_{tt}=u_{\,|i}^{k}u_{\,|k}^{i}=\nabla u:\nabla u\label{eq:3-30-1}
\end{equation}
By (\ref{eq:A-15})
\begin{alignat*}{1}
J_{tn} & =g^{ik}g^{jl}\left\langle (D_{E_{i}}u)^{\top},E_{j}\right\rangle \left\langle (D_{E_{l}}v)^{\top},E_{k}\right\rangle \\
 & =g^{ik}g^{jl}\left\langle \nabla_{E_{i}}u,E_{j}\right\rangle \left\langle fD_{E_{l}}N,E_{k}\right\rangle \\
 & =-fg^{ik}g^{jl}\left\langle u_{\,|i}^{r}E_{r},E_{j}\right\rangle II_{M}(E_{l},E_{k})\\
 & =-fg^{ik}u_{\,|i}^{l}B_{lk}=-fu_{\,|i}^{l}B_{\,l}^{i}
\end{alignat*}
and
\begin{equation}
J_{tn}=-fu_{\,|i}^{l}B_{\,l}^{i}=-fII_{M}:\nabla u\label{eq:3-30-2}
\end{equation}
The symmetry $J_{tn}=J_{nt}$ follows immediately by interchanging
indices $i\leftrightarrow l$, $j\leftrightarrow k$. Again by (\ref{eq:A-15})
\begin{alignat*}{1}
J_{nn} & =g^{ik}g^{jl}\left\langle (D_{E_{i}}v)^{\top},E_{j}\right\rangle \left\langle (D_{E_{l}}v)^{\top},E_{k}\right\rangle \\
 & =f^{2}g^{ik}g^{jl}\left\langle D_{E_{i}}N,E_{j}\right\rangle \left\langle D_{E_{l}}N,E_{k}\right\rangle \\
 & =f^{2}g^{ik}g^{jl}\left\langle N,D_{E_{i}}E_{j}\right\rangle \left\langle N,D_{E_{l}}E_{k}\right\rangle =f^{2}B_{\,i}^{k}B_{\,k}^{i}
\end{alignat*}
and
\begin{equation}
J_{nn}=f^{2}B_{\,i}^{k}B_{\,k}^{i}=f^{2}|B_{M}|^{2}.\label{eq:3-30-3}
\end{equation}

By (see \cite{key-6})
\begin{equation}
\mathrm{div}_{M}w=\mathrm{div}_{M}u-H\cdot v\label{eq:A-16}
\end{equation}
since $M$ was supposed to have constant mean curvature $\kappa=H\cdot N$,
we obtain $\mathrm{div}_{M}w=\mathrm{div}_{M}u-\kappa f$ and from
this
\begin{equation}
(\mathrm{div}_{M}w)^{2}=(\mathrm{div}_{M}u)^{2}-2\kappa f\mathrm{div}_{M}u+\kappa^{2}f^{2}\label{eq:3-31}
\end{equation}

After this preparation we are ready to calculate the right side of
(\ref{eq:A-14}). By (\ref{eq:3-28}), (\ref{eq:3-30-2}) and the
properties of covariant derivative we obtain
\begin{alignat}{1}
I_{tn}-J_{tn} & =B_{\,i}^{j}u^{i}\frac{\partial f}{\partial q^{j}}+fu_{\,|j}^{i}B_{\,i}^{j}=B_{\,i}^{j}(fu^{i})_{|j}\nonumber \\
 & =(fB_{\,i}^{j}u^{i})_{|j}-fu^{i}B_{\,i|j}^{j}\label{eq:A-17}
\end{alignat}
We will prove that for a hypersurface with constant mean curvature
\begin{equation}
B_{\,i|j}^{j}=0.\label{eq:3-41}
\end{equation}
By the Mainardi-Codazzi equations in component form
\[
B_{jk|i}=B_{ik|j}
\]
and Ricci's lemma, we obtain
\[
B_{\,j|i}^{k}=B_{\,i|j}^{k}
\]
and from this by contraction over the indices $i$, $k$
\begin{equation}
B_{\,j|i}^{i}=B_{\,i|j}^{i}.\label{eq:3-45}
\end{equation}
The Mainardi-Codazzi equations in component form were obtained from
their component-free version (\cite{key-8}, vol. III, p. 10),
\[
(\nabla_{X}II)(Y,Z)=(\nabla_{Y}II)(X,Z)
\]
by setting $X=E_{i}$, $Y=E_{j}$, $Z=E_{k}$ and recalling that
\[
(\nabla_{E_{i}}II)(E_{j},E_{k})=B_{jk|i}.
\]
By the definition of mean curvature (\ref{eq:def-mean-curv}) and
equations (\ref{eq:def-H}) and (\ref{eq:def-mean-2nd-ff}),
\begin{equation}
\kappa=B_{\,i}^{i}.\label{eq:3-39-1}
\end{equation}
Since $\kappa$ is constant by hypothesis, application of (\ref{eq:3-45})
yields (\ref{eq:3-41}), and with this (\ref{eq:A-17}) reduces to
\begin{equation}
I_{tn}-J_{tn}=(fB_{\,i}^{j}u^{i})_{|j}=-\mathrm{div}_{M}(fWu)\label{eq:3-40}
\end{equation}

We calculate the third item under the integral sign on the right side
of (\ref{eq:A-14}). By (\ref{eq:3-31}) and (\ref{eq:3-30-1}) we
have
\begin{equation}
L-J_{tt}=(u_{\,i}^{i}u^{j}-u^{i}u_{\,|i}^{j})_{|j}+u^{i}(u_{\,|ij}^{j}-u_{\,ji}^{j})-2\kappa f\mathrm{div}_{M}u+\kappa^{2}f^{2}\label{eq:A-20}
\end{equation}
\begin{equation}
\begin{alignedat}{1}L-J_{tt} & =u_{\,|i}^{i}u_{\,|j}^{j}-u_{\,|j}^{i}u_{\,|i}^{j}-2\kappa fu_{\,|i}^{i}+\kappa^{2}f^{2}\\
 & =(u_{\,|i}^{i}u^{j}-u^{i}u_{\,|i}^{j})_{|j}-u^{j}u_{\,|ij}^{i}+u^{i}u_{\,|ij}^{j}-2\kappa fu_{\,|i}^{i}+\kappa^{2}f^{2}\\
 & =(u_{\,|i}^{i}u^{j}-u^{i}u_{\,|i}^{j})_{|j}+u^{i}(u_{\,|ij}^{j}-u_{\,|ji}^{j})-2\kappa fu_{\,|i}^{i}+\kappa^{2}f^{2}
\end{alignedat}
\label{eq:3-53}
\end{equation}
By Ricci's identity (see \cite{key-8}, vol. II, p. 224),
\[
u_{\,|ij}^{r}-u_{\,|ji}^{r}=R_{\,kji}^{r}u^{k}
\]
on contracting upon the indices $r,i$
\[
u_{\,|ij}^{i}-u_{\,|ji}^{i}=R_{\,kji}^{i}u^{k}=R_{kj}u^{k}
\]
and renaming indices, we obtain 
\begin{equation}
u_{\,|ij}^{j}-u_{\,|ji}^{j}=-R_{ki}u^{k}\label{eq:3-54}
\end{equation}
In these equalities
\[
R_{jk}=R_{\,jki}^{i}=g^{ri}R_{rjki}
\]
are the components of Ricci's tensor and
\[
R_{ijkl}=\left\langle R(E_{k},E_{l})E_{j},E_{i}\right\rangle 
\]
are the components of the curvature tensor (\cite{key-8}, vol. II,
pp. 190, 239). By Gauss' Theorema Egregium (\cite{key-8}, vol III,
p. 5, Theorem 6) we have
\[
\left\langle R(E_{k},E_{l})E_{j},E_{i}\right\rangle =B_{ik}B_{jl}-B_{il}B_{jk}
\]
hence
\begin{equation}
R_{jk}=g^{il}B_{ik}B_{jl}-B_{\,i}^{i}B_{jk}.\label{eq:3-46}
\end{equation}
Since $M$ has constant mean curvature, by (\ref{eq:3-39-1}) it follows
that
\begin{equation}
R_{jk}=g^{rs}B_{rj}B_{sk}-\kappa B_{jk}\label{eq:3-47}
\end{equation}
Using this equality in (\ref{eq:3-54}) gives
\begin{equation}
u^{i}(u_{\,|ij}^{j}-u_{\,|ji}^{j})=-g^{rs}B_{ri}B_{sk}u^{i}u^{k}+\kappa B_{ik}u^{i}u^{k}\label{eq:3-54-1}
\end{equation}
Combination of (\ref{eq:3-53}), (\ref{eq:3-54-1}) and (\ref{eq:3-26})
yields
\begin{equation}
\begin{alignedat}{1}L+I_{tt}-J_{tt} & =\mathrm{div}_{M}\left(u\mathrm{div}_{M}u-\nabla_{u}u\right)\\
 & +\kappa II_{M}(u,u)-2\kappa f\mathrm{div}_{M}u+\kappa^{2}f^{2}
\end{alignedat}
\label{eq:3-51}
\end{equation}

On using (\ref{eq:3-51}), (\ref{eq:3-40}), (\ref{eq:3-25-1}) and
(\ref{eq:3-30-3}) in (\ref{eq:A-14}) we obtain
\begin{alignat*}{1}
\delta^{2}A(M) & =\int_{M}\left(|\mathrm{grad}_{M}f|^{2}-|B_{M}|^{2}f^{2}\right)\\
 & -2\int_{M}\mathrm{div}_{M}(fWu)+\int_{M}\mathrm{div}_{M}\left(u\mathrm{div}_{M}u-\nabla_{u}u\right)\\
 & +\int_{M}\kappa\left[II_{M}(u,u)-2f\mathrm{div}_{M}u+\kappa f^{2}\right]+\int_{M}\left(\mathrm{div}_{M}a^{\top}-H\cdot a^{\bot}\right)
\end{alignat*}
Application of the divergence theorem gives
\begin{alignat*}{1}
\delta^{2}A(M) & =\int_{M}\left(|\mathrm{grad}_{M}f|^{2}-|B_{M}|^{2}f^{2}\right)\\
 & +\int_{M}\kappa\left[II_{M}(u,u)-2f\mathrm{div}_{M}u+\kappa f^{2}-a\cdot N\right]\\
 & -2\int_{\partial M}f\left\langle Wu,\nu\right\rangle +\int_{\partial M}\left[(u\cdot\nu)\mathrm{div}_{M}u-\left\langle \nabla_{u}u,\nu\right\rangle \right]+\int_{\partial M}a\cdot\nu
\end{alignat*}
From this, equality (\ref{eq:3-19}) follows immediately from
\begin{equation}
\left\langle Wu,\nu\right\rangle =-II_{M}(u,\nu).\label{eq:3-26-1}
\end{equation}
Formula (\ref{eq:3-19-1}) follows by Lemma \ref{lem:3-1.9} below.
\end{proof}
\begin{lem}
\label{lem:3-1.9}For a variation of type (\ref{eq:3-4}) the second
variation field $a$ satisfies the following equality
\begin{equation}
\int_{M}\mathrm{div}_{M}a^{\top}=\int_{\partial M}\left(f\left\langle D_{N}w,\nu\right\rangle -fII_{M}(u,\nu)+\left\langle \nabla_{u}u,\nu\right\rangle \right)\label{eq:3-33-1}
\end{equation}
\end{lem}

\begin{proof}
Application of the divergence theorem to (\ref{eq:3-4}) gives
\begin{equation}
\int_{M}\mathrm{div}_{M}a^{\top}=\int_{\partial M}\left\langle D_{w}w,\nu\right\rangle =\int_{\partial M}f\left\langle D_{N}w,\nu\right\rangle +\int_{\partial M}\left\langle D_{u}w,\nu\right\rangle .\label{eq:3-36}
\end{equation}
By the decomposition $w=u+v$ and
\begin{equation}
\left\langle D_{u}v,\nu\right\rangle =-\left\langle v,D_{u}\nu\right\rangle =-fII_{M}(u,\nu)\label{eq:3-39}
\end{equation}
we obtain (\ref{eq:3-33-1}).
\end{proof}
\begin{lem}
Let $T=(T_{j})_{j=1}^{r}$, $T_{j}=(M_{pj})_{p=1}^{3}$, be a three
phase partitioning of a domain $\Omega\subset\mathbb{R}^{n}$ by a
system of $r$ $C^{2}$-triple junctions into the domains $\Omega_{pj}$
$(p=1,\cdots,3;j=1,\cdots,r)$. Let $\tilde{\Omega}=\Omega_{pj}$
be any one of these domains. Then, for variations $w$ of type (\ref{eq:3-4})
preserving $\Sigma$, the second variation of volume of $\tilde{\Omega}$
is given by
\begin{equation}
\delta^{2}|\tilde{\Omega}|=\int_{\partial\tilde{\Omega}\backslash\Sigma}(w\cdot N_{\partial\tilde{\Omega}})\mathrm{div}_{\mathbb{R}^{n}}w\label{eq:A-19}
\end{equation}
where $N_{\partial\tilde{\Omega}}$ is the unit outward normal of
$\tilde{\Omega}$.
\end{lem}

\begin{proof}
If $(\xi^{t})_{t\in I}$ is a variation with first variation field
$w$, and $\tilde{\Omega}^{t}=\xi^{t}(\tilde{\Omega})$, we have
\[
|\tilde{\Omega}^{t}|=\int_{\xi^{t}(\tilde{\Omega})}dx=\int_{\tilde{\Omega}}J\xi^{t}(y)dy
\]
where $J\xi^{t}$ is the Jacobian of $\xi^{t}$. For the second variation
of this functional we have 
\[
\delta^{2}|\tilde{\Omega}|=\left.\frac{d^{2}}{dt^{2}}|\tilde{\Omega}^{t}|\right|_{t=0}=\int_{\tilde{\Omega}}\left.\frac{\partial^{2}}{\partial t^{2}}J\xi^{t}(y)\right|_{t=0}dy.
\]
Application of the rule of determinant differentiation and straight-forward
manipulations give
\[
\left.\frac{\partial^{2}}{\partial t^{2}}J\xi^{t}(y)\right|_{t=0}=\mathrm{div}_{\mathbb{R}^{n}}a+\frac{\partial w^{\alpha}}{\partial x^{\alpha}}\frac{\partial w^{\beta}}{\partial x^{\beta}}-\frac{\partial w^{\alpha}}{\partial x^{\beta}}\frac{\partial w^{\beta}}{\partial x^{\alpha}}.
\]
We are using Greek indices for vector components and coordinates in
the surrounding space $\mathbb{R}^{n}$ and Latin for the manifold
$M$. Summation convention applies to Greek indices as well. Formula
(\ref{eq:A-19}) follows from this equality, the identity
\begin{alignat*}{1}
\frac{\partial w^{\alpha}}{\partial x^{\alpha}}\frac{\partial w^{\beta}}{\partial x^{\beta}}-\frac{\partial w^{\alpha}}{\partial x^{\beta}}\frac{\partial w^{\beta}}{\partial x^{\alpha}} & =\frac{\partial}{\partial x^{\alpha}}\left(w^{\alpha}\frac{\partial w^{\beta}}{\partial x^{\beta}}-w^{\beta}\frac{\partial w^{\alpha}}{\partial x^{\beta}}\right)\\
 & =\mathrm{div}_{\mathbb{R}^{n}}\left(w\mathrm{div}_{\mathbb{R}^{n}}w-D_{w}w\right)
\end{alignat*}
and Gauss' theorem, in view of (\ref{eq:A-13}). The hypothesis that
the variation preserves $\Sigma$ is only used to drop the integral
over $\partial\tilde{\Omega}\cap\Sigma$. 
\end{proof}
On the basis of the above results, in the next theorem we develop
a formula for the second variation of area for minimal triple junction
partitions with volume constraints.
\begin{thm}
\label{thm:sec-var-TJ}Let $\Omega$ be a domain in $\mathbb{R}^{3}$,
$T=(T_{j})_{j=1}^{r}$, $T_{j}=(M_{pj})_{p=1}^{3}$, a minimal three-phase
partition of $\Omega$ by a set of $r$ $C^{2}$ triple junctions
with volume constraints, and $w$ an admissible variation satisfying
(\ref{eq:3-4}) and the volume constraints. On each leaf $M_{pj}$
we have the splitting $w=u_{pj}+v_{pj}$, $u_{pj}\in TM_{pj}$, $v_{pj}\in NM_{pj}$,
and we set $f_{pj}=w\cdot N_{pj}=v_{pj}\cdot N_{pj}$, $N_{pj}$ being
the unit normal field of $M_{pj}$. Then the following formula holds
for the second variation of the area functional,
\begin{equation}
\begin{gathered}\delta^{2}A^{\star}(T)=\sum_{p,j}\gamma_{pj}\int_{M_{pj}}\left(|\mathrm{grad}_{M_{pj}}f_{pj}|^{2}-|B_{M_{pj}}|^{2}f_{pj}^{2}\right)\\
-\sum_{j=1}^{r}\sum_{p=1}^{3}\gamma_{pj}\int_{\partial M_{pj}\cap\varSigma}f_{pj}^{2}II_{\varSigma}(N_{pj},N_{pj})\\
+\sum_{j=1}^{r}\sum_{p=1}^{3}\gamma_{pj}\int_{S_{j}}f_{pj}h_{pj}II_{M_{pj}}(\nu_{pj},\nu_{pj})
\end{gathered}
\label{eq:A-36}
\end{equation}
where $S_{j}$ is the spine of $T_{j}$ and $\nu_{pj}\in TM_{pj}$
is the unit normal field of $\partial M_{pj}\cap S_{j}$.
\end{thm}

\begin{proof}
The second variation of the area functional with Lagrange multipliers,
equation (\ref{eq:modif-area-func}), is given by
\[
\delta^{2}A^{\star}(M)=\delta^{2}A(M)-\sum_{j=2}^{3}\lambda_{j}\sum_{k=1}^{P_{j}}\delta^{2}\left|\Omega_{jk}\right|
\]
Again, for concreteness and to keep the length of formulas to a minimum,
we consider the disconnected three phase partitioning of Figure \ref{fig:discon-d-tj}
with the indicated orientation. For this configuration the area functional
is given by (\ref{eq:modif-area-func}). Use of (\ref{eq:A-19}) gives
\[
\begin{alignedat}{1}\delta^{2}A^{\star}(M) & =\sum_{j=1}^{2}\sum_{p=1}^{3}\gamma_{pj}\delta^{2}A(M_{pj})-\lambda_{2}\delta^{2}\left|\Omega_{2}\right|-\lambda_{3}\delta^{2}\left|\Omega_{3}\right|\\
 & =\sum_{j=1}^{2}\sum_{p=1}^{3}\gamma_{pj}\delta^{2}A(M_{pj})-(\lambda_{2}-\lambda_{3})\int_{M_{1}}w\cdot N_{1}\,\mathrm{div}_{\mathbb{R}^{3}}w\\
 & -\lambda_{2}\int_{M_{31}}w\cdot(-N_{31})\mathrm{div}_{\mathbb{R}^{3}}w-\lambda_{2}\int_{M_{32}}w\cdot(-N_{32})\mathrm{div}_{\mathbb{R}^{3}}w\\
 & -\lambda_{3}\int_{M_{21}}w\cdot N_{21}\,\mathrm{div}_{\mathbb{R}^{3}}w-\lambda_{3}\int_{M_{22}}w\cdot N_{22}\,\mathrm{div}_{\mathbb{R}^{3}}w
\end{alignedat}
\]
By equations (\ref{eq:3.14}) and $\kappa_{pj}=H_{pj}\cdot N_{pj}$
(see Proposition \ref{prop:1st-var-gen}(ii)\-) we obtain
\[
\begin{alignedat}{1}\delta^{2}A^{\star}(M) & =\sum_{j=1}^{2}\sum_{p=1}^{3}\gamma_{pj}\delta^{2}A(M_{pj})+\gamma_{1}\kappa_{1}\int_{M_{1}}w\cdot N_{1}\,\mathrm{div}_{\mathbb{R}^{3}}w\\
 & +\gamma_{31}\kappa_{31}\int_{M_{31}}w\cdot N_{31}\,\mathrm{div}_{\mathbb{R}^{3}}w+\gamma_{32}\kappa_{32}\int_{M_{32}}w\cdot N_{32}\,\mathrm{div}_{\mathbb{R}^{3}}w\\
 & +\gamma_{21}\kappa_{21}\int_{M_{21}}w\cdot N_{21}\,\mathrm{div}_{\mathbb{R}^{3}}w+\gamma_{22}\kappa_{22}\int_{M_{22}}w\cdot N_{22}\,\mathrm{div}_{\mathbb{R}^{3}}w
\end{alignedat}
\]
Each term of the double sum on the right side corresponds to one integral,
and thus we can express $\delta^{2}A^{\star}(M)$ as follows:
\[
\delta^{2}A^{\star}(M)=\sum_{j=1}^{2}\sum_{p=1}^{3}\gamma_{pj}\left[\delta^{2}A(M_{pj})+\kappa_{pj}\int_{M_{pj}}w\cdot N_{pj}\,\mathrm{div}_{\mathbb{R}^{3}}w\right]
\]

We set
\[
\delta^{2}A^{\star}(M_{pj})=\delta^{2}A(M_{pj})+\kappa_{pj}\int_{M_{pj}}w\cdot N_{pj}\,\mathrm{div}_{\mathbb{R}^{3}}w
\]
so that
\[
\delta^{2}A^{\star}(M)=\sum_{j}\sum_{p=1}^{3}\gamma_{pj}\delta^{2}A^{\star}(M_{pj}).
\]
Using formula (\ref{eq:3-19}) of Proposition \ref{prop:3-1.3} gives
for $\delta^{2}A^{\star}(M_{pj})$
\[
\begin{alignedat}{1}\delta^{2}A^{\star}(M_{pj}) & =\int_{M_{pj}}\left(|\mathrm{grad}_{M_{pj}}f|^{2}-|B_{M_{pj}}|^{2}f^{2}\right)\\
 & +\int_{M_{pj}}\kappa_{pj}\left[II_{M_{pj}}(u,u)-2f\mathrm{div}_{M_{pj}}u+\kappa_{pj}f^{2}-a\cdot N_{pj}\right]\\
 & +\int_{\partial M_{pj}}\left[(u\cdot\nu)\mathrm{div}_{M_{pj}}u-\left\langle \nabla_{u}u,\nu\right\rangle +2fII_{M_{pj}}(u,\nu)+a\cdot\nu\right]\\
 & +\kappa_{pj}\int_{M_{pj}}w\cdot N_{pj}\,\mathrm{div}_{\mathbb{R}^{3}}w
\end{alignedat}
\]
and reordering terms,
\begin{equation}
\begin{alignedat}{1}\delta^{2}A^{\star}(M_{pj}) & =\int_{M_{pj}}\left(|\mathrm{grad}_{M_{pj}}f|^{2}-|B_{M_{pj}}|^{2}f^{2}\right)\\
 & +\int_{M_{pj}}\kappa_{pj}\left[II_{M_{pj}}(u,u)-2f\mathrm{div}_{M_{pj}}u+\kappa_{pj}f^{2}\right]\\
 & +\int_{\partial M_{pj}}\left[(u\cdot\nu)\mathrm{div}_{M_{pj}}u-\left\langle \nabla_{u}u,\nu\right\rangle +2fII_{M_{pj}}(u,\nu)\right.\\
 & \left.\qquad\qquad+a\cdot\nu\right]\\
 & +\int_{M_{pj}}\kappa_{pj}\left[-\left\langle D_{w}w,N_{pj}\right\rangle +\left(w\cdot N_{pj}\right)\,\mathrm{div}_{\mathbb{R}^{3}}w\right]
\end{alignedat}
\label{eq:A-23}
\end{equation}
We treat the last two items under the integral sign in the last line.
For brevity we drop the indices $p,j$ on $N_{pj}$ and $M_{pj}$.
Substituting $w=u+v$, $v=fN$ and expanding, we obtain 
\[
\begin{alignedat}{1}\left\langle D_{w}w-(\mathrm{div}_{\mathbb{R}^{n}}w)w,N\right\rangle  & =\left\langle D_{u}u,N\right\rangle +\left\langle D_{u}v,N\right\rangle +\left\langle D_{v}u,N\right\rangle \\
 & +\left\langle D_{v}v,N\right\rangle -f\mathrm{div}_{\mathbb{R}^{n}}w\\
 & =II_{M}(u,u)+\left\langle fD_{u}N+(D_{u}f)N,N\right\rangle \\
 & +f\left\langle D_{N}u,N\right\rangle +f\left\langle D_{N}v,N\right\rangle \\
 & -f\mathrm{div}_{\mathbb{R}^{n}}v-f\mathrm{div}_{\mathbb{R}^{n}}u
\end{alignedat}
\]
By the identity
\[
\mathrm{div}_{\mathbb{R}^{n}}X=\mathrm{div}_{M}X+\left\langle D_{N}X,N\right\rangle 
\]
and dropping canceling terms we obtain
\begin{equation}
\left\langle D_{w}w-(\mathrm{div}_{\mathbb{R}^{n}}w)w,N\right\rangle =II_{M}(u,u)+\left\langle u,\mathrm{grad}_{M}f\right\rangle -f\mathrm{div}u+\kappa f^{2}\label{eq:A-24}
\end{equation}
Combination of (\ref{eq:A-23}) and (\ref{eq:A-24}) gives
\[
\begin{alignedat}{1}\delta^{2}A^{\star}(M_{pj}) & =\int_{M_{pj}}\left(|\mathrm{grad}_{M_{pj}}f|^{2}-|B_{M_{pj}}|^{2}f^{2}\right)\\
 & +\int_{\partial M_{pj}}\left\langle u\mathrm{div}u-\nabla_{u}u-2fWu-\kappa_{pj}fu+a,\nu\right\rangle 
\end{alignedat}
\]
Use of equation (\ref{eq:3-33-1}) and simplification gives
\[
\begin{alignedat}{1}\delta^{2}A^{\star}(M_{pj}) & =\int_{M_{pj}}\left(|\mathrm{grad}_{M_{pj}}f|^{2}-|B_{M_{pj}}|^{2}f^{2}\right)\\
 & +\int_{\partial M_{pj}}\left\langle u\mathrm{div}u-fWu-\kappa_{pj}fu+fD_{N_{pj}}w,\nu\right\rangle 
\end{alignedat}
\]

We break integrals over $\partial M$ (again for brevity we drop indices
on leaves and related quantities) into integrals over $\partial M\cap\varSigma$
and $\partial M\cap\Omega=S$, where $S$ is the spine of the triple
junction $M$ belongs to. Since by minimality $u\cdot\nu=0$ (see
equation (\ref{eq:3.15-1})\-), 
\begin{equation}
\begin{gathered}\int_{\partial M}\left[fII_{M}(u,\nu)+\left(\mathrm{div}u-\kappa f\right)\left\langle u,\nu\right\rangle +f\left\langle D_{N}w,\nu\right\rangle \right]=\\
\int_{S}\left[fII_{M}(u,\nu)+\left(\mathrm{div}u-\kappa f\right)\left\langle u,\nu\right\rangle +f\left\langle D_{N}w,\nu\right\rangle \right]\\
+\int_{\partial M\cap\varSigma}f\left[II_{M}(u,\nu)+\left\langle D_{N}w,\nu\right\rangle \right]
\end{gathered}
\label{eq:100}
\end{equation}
Let $\tau$ be a unit tangent field of $\partial M\cap\varSigma$.
The triple $\nu,\tau,N$ makes up an orthonormal frame along $\partial M\cap\varSigma$.
Since $u\cdot\nu=u\cdot N=0$ on $\partial M\cap\varSigma$, there
is a $C^{1}$ function $g:\partial M\cap\varSigma\to\mathbb{R}$ such
that $u=g\tau$ on $\partial M\cap\varSigma$, and as a consequence
of this and $w\cdot\nu=u\cdot\nu=0$, we have $w=fN+g\tau$ on $\partial M\cap\varSigma$.
Thus, for the boundary part of the above integral we have
\begin{equation}
\begin{gathered}\int_{\partial M\cap\varSigma}f\left[II_{M}(u,\nu)+\left\langle D_{N}w,\nu\right\rangle \right]=\\
\int_{\partial M\cap\varSigma}fII_{M}(u,\nu)-\int_{\partial M\cap\varSigma}f\left\langle w,D_{N}\nu\right\rangle =\\
\int_{\partial M\cap\varSigma}fgII_{M}(\tau,\nu)-\int_{\partial M\cap\varSigma}f^{2}\left\langle N,D_{N}\nu\right\rangle -\int_{\partial M\cap\varSigma}fg\left\langle \tau,D_{N}\nu\right\rangle 
\end{gathered}
\label{eq:101}
\end{equation}
Let $N_{\varSigma}$ be the \emph{inward pointing} unit normal field
of $\varSigma$. By
\[
II_{M}(\tau,\nu)=\left\langle D_{\tau}\nu,N\right\rangle =-\left\langle \nu,D_{\tau}N\right\rangle =\left\langle N_{\Sigma},D_{\tau}N\right\rangle =II_{\Sigma}(\tau,N),
\]
since $\tau,N\in T\Sigma$, and
\[
\left\langle \tau,D_{N}\nu\right\rangle =-\left\langle \tau,D_{N}N_{\varSigma}\right\rangle =\left\langle D_{N}\tau,N_{\varSigma}\right\rangle =II_{\varSigma}(\tau,N),
\]
the first and last term on the last row of (\ref{eq:101}) cancel
out, and we obtain
\[
\begin{gathered}\int_{\partial M\cap\varSigma}f\left[II_{M}(u,\nu)+\left\langle D_{N}w,\nu\right\rangle \right]=\int_{\partial M\cap\varSigma}f^{2}\left\langle D_{N}N,\nu\right\rangle \\
=-\int_{\partial M\cap\varSigma}f^{2}II_{\varSigma}(N,N).
\end{gathered}
\]
Multiplication by $\gamma_{pj}$ and summation over the possible values
of $p,j$ (i.e. over the leaves intersecting the boundary of $\Omega$)
gives the first term on the second row of (\ref{eq:A-36}).

Finally, we treat the integral over the spine $S$ in (\ref{eq:100}).
Recall that we are dropping leaf indices, and $\mathrm{div}u-\kappa f=\mathrm{div}_{M}w$.
We consider as previously the unit tangent field $\tau$ of $S$,
so that the triple $\nu,\tau,N$ makes up an orthonormal frame along
$S$. Again, there are $C^{1}$ functions $h,g:S\to\mathbb{R}$ such
that $u=h\nu+g\tau$ on $S$, and as a consequence of this, $w=fN+h\nu+g\tau$
on $S$. From the identities
\[
\mathrm{div}_{M}w=\mathrm{div}_{\mathbb{R}^{3}}w-\left\langle D_{N}w,N\right\rangle 
\]
\[
\mathrm{div}_{S}w=\mathrm{div}_{\mathbb{R}^{3}}w-\left\langle D_{N}w,N\right\rangle -\left\langle D_{\nu}w,\nu\right\rangle 
\]
we obtain
\[
\mathrm{div}_{M}w=\mathrm{div}_{S}w+\left\langle D_{\nu}w,\nu\right\rangle .
\]
Using this equality, the quantity under the integral sign over $S$
on the second row of (\ref{eq:100}) assumes the form
\[
\begin{aligned} & (u\cdot\nu)\mathrm{div}_{M}w+fII_{M}(u,\nu)+f\left\langle D_{N}w,\nu\right\rangle =\\
 & (w\cdot\nu)\mathrm{div}_{S}w+\left\langle D_{\nu}w,\nu\right\rangle h+fII_{M}(u,\nu)+\left\langle D_{v}w,\nu\right\rangle =\\
 & (w\cdot\nu)\mathrm{div}_{S}w+\left\langle D_{w}w,\nu\right\rangle -g\left\langle D_{\tau}w,\nu\right\rangle +fII_{M}(u,\nu)=\\
 & (w\cdot\nu)\mathrm{div}_{S}w+\left\langle D_{w}w,\nu\right\rangle -g\left\langle D_{\tau}w,\nu\right\rangle +fhII_{M}(\nu,\nu)+fgII_{M}(\tau,\nu)
\end{aligned}
\]
Multiplication by $\gamma_{pj}$ and summation over the leaves of
$T_{j}$, in view of (\ref{eq:3.10-1}) gives
\begin{equation}
\sum_{p=1}^{3}\gamma_{pj}f_{pj}h_{pj}II_{M_{pj}}(\nu_{pj},\nu_{pj})+g_{j}\sum_{p=1}^{3}\gamma_{pj}f_{pj}II_{M_{pj}}(\tau_{j},\nu_{pj})\label{eq:102}
\end{equation}
We will prove that the second term vanishes. On each leaf (dropping
indices) we have
\[
\begin{alignedat}{1}fII_{M}(\tau,\nu) & =-f\left\langle D_{\tau}N,\nu\right\rangle =-\left\langle D_{\tau}v,\nu\right\rangle \\
 & =-\left\langle D_{\tau}w,\nu\right\rangle +\left\langle D_{\tau}(h\nu+g\tau),\nu\right\rangle \\
 & =-\left\langle D_{\tau}w,\nu\right\rangle +g\left\langle D_{\tau}\tau,\nu\right\rangle +D_{\tau}h.
\end{alignedat}
\]
Multiplication by $\gamma_{pj}$ and summation over the leaves of
$T_{j}$, in view of (\ref{eq:3.10-1}) and
\[
\sum_{p=1}^{3}\gamma_{pj}h_{pj}=\sum_{p=1}^{3}\gamma_{pj}w\cdot\nu_{pj}=0
\]
gives
\[
\sum_{p=1}^{3}\gamma_{pj}f_{pj}II_{M_{pj}}(\tau_{j},\nu_{pj})=0.
\]
Summation of the surviving term $\sum_{p=1}^{3}\gamma_{pj}f_{pj}h_{pj}II_{M_{pj}}(\nu_{pj},\nu_{pj})$
in (\ref{eq:102}) over all triple junctions present in the system
gives the second term on the second row of equation (\ref{eq:A-36}),
and this completes the proof.
\end{proof}
\begin{rem}
The normal field of $\Sigma$, $N_{\Sigma}$, was chosen so that the
component $II_{\varSigma}(N_{pj},N_{pj})$ of the second fundamental
form is non-negative for convex $\Omega$.
\end{rem}

\begin{rem}
For two phase partitions and normal variations, (\ref{eq:A-36}) reduces
to the second variation formula in \cite{key-7}.
\end{rem}

For subsequent reference, we summarize the matching conditions on
the spine of minimal triple junctions. Let $T=(M_{i})_{i=1}^{3}$
be such a triple junction with spine $S=\partial M_{1}=\partial M_{2}=\partial M_{3}$.
For simplicity, we assume $\gamma_{1}=\gamma_{2}=\gamma_{3}=1$. By
the first variation, Proposition \ref{prop:1st-var-gen}, we have
(refer to Figure \ref{fig:2})
\begin{equation}
\nu_{1}+\nu_{2}+\nu_{3}=0\label{eq:3-73-1}
\end{equation}
\begin{equation}
N_{1}+N_{2}+N_{3}=0\label{eq:3-73-2}
\end{equation}
on $S$. For any vector $X\in T_{p}\mathbb{R}^{3}$, $p\in S$, its
projections $X_{i}=X\cdot\nu_{i}$ on the $\nu_{i}$ satisfy
\begin{equation}
X_{1}+X_{2}+X_{3}=0\label{eq:3-74}
\end{equation}
A similar equality is satisfied by the projections of $X$ on $N_{i}$.

\begin{figure}[h]
\includegraphics[scale=0.75]{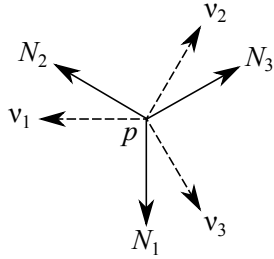}

\caption{\label{fig:2}The unit vector fields $N_{i}$, $\nu_{i}$ ($i=1,2,3)$
at $p\in S$. $S$ is the spine of the triple junction. $N_{i}$ is
normal to $M_{i}$ and $\nu_{i}\in TM_{i}$ is normal to $T_{p}S$
(cf. Figure \ref{fig:Intro-st-unst-TJs}, triple junction on the left).
The plane of the sheet is perpendicular to the tangent vector $\tau(p)\in T_{p}S$
(not shown).}

\end{figure}

For the first variation field $w$ we use the following notation:
\begin{equation}
\left.\begin{array}{l}
f_{i}=v_{i}\cdot N_{i}=w\cdot N_{i}\\
h_{i}=u_{i}\cdot\nu_{i}=w\cdot\nu_{i}\\
g_{i}=u_{i}\cdot\tau_{i}=w\cdot\tau_{i}
\end{array}\right\} \quad i=1,2,3\label{eq:3-75}
\end{equation}
where $\tau$ is the unit tangent field of the spine such that at
any point $p\in S$ the triple $(\tau,\nu_{i},N_{i})$ is positively
oriented for all $i=1,2,3$. By (\ref{eq:3-74}) we have
\begin{equation}
\begin{array}{l}
f_{1}+f_{2}+f_{3}=0\\
h_{1}+h_{2}+h_{3}=0\\
g_{1}=g_{2}=g_{3}=g
\end{array}\label{eq:3-76}
\end{equation}
From Figure \ref{fig:2} we obtain the following elementary geometric
relations:
\begin{equation}
\begin{array}{l}
N_{2}=-\frac{1}{2}N_{1}-\frac{\sqrt{3}}{2}\nu_{1}\\
\nu_{2}=+\frac{\sqrt{3}}{2}N_{1}-\frac{1}{2}\nu_{1}\\
N_{3}=-\frac{1}{2}N_{1}+\frac{\sqrt{3}}{2}\nu_{1}\\
\nu_{3}=-\frac{\sqrt{3}}{2}N_{1}-\frac{1}{2}\nu_{1}
\end{array}\label{eq:3-77}
\end{equation}
From these relations and (\ref{eq:3-75}) we obtain
\begin{equation}
\begin{array}{l}
f_{2}=-\frac{1}{2}f_{1}-\frac{\sqrt{3}}{2}h_{1}\\
h_{2}=+\frac{\sqrt{3}}{2}f_{1}-\frac{1}{2}h_{1}\\
f_{3}=-\frac{1}{2}f_{1}+\frac{\sqrt{3}}{2}h_{1}\\
h_{3}=-\frac{\sqrt{3}}{2}f_{1}-\frac{1}{2}h_{1}
\end{array}\label{eq:3-78}
\end{equation}
Equations (\ref{eq:3-77}), (\ref{eq:3-78}) express the \emph{matching
conditions on the spine}. When $\gamma_{1}\neq\gamma_{2}\neq\gamma_{3}\neq\gamma_{1}$
we have again linear dependences similar to (\ref{eq:3-77}), (\ref{eq:3-78}),
with coefficients depending on $\gamma_{1},\gamma_{2},\gamma_{3}$.
\begin{cor}
In the setting of Theorem \ref{thm:sec-var-TJ} and assuming that
$\gamma_{1}=\gamma_{2}=\gamma_{3}=1$, the expression for the second
variation of area (\ref{eq:A-36}) reduces to
\begin{equation}
\begin{gathered}\delta^{2}A^{\star}(T)=\sum_{p,j}\int_{M_{pj}}\left(|\mathrm{grad}_{M_{pj}}f_{pj}|^{2}-|B_{M_{pj}}|^{2}f_{pj}^{2}\right)\\
-\sum_{j=1}^{r}\sum_{p=1}^{3}\int_{\partial M_{pj}\cap\varSigma}f_{pj}^{2}II_{\varSigma}(N_{pj},N_{pj})\\
+\sum_{j=1}^{r}\int_{S_{j}}\left[\alpha_{j}\left(f_{1j}^{2}-h_{1j}^{2}\right)+2\beta_{j}f_{1j}h_{1j}\right]
\end{gathered}
\label{eq:103}
\end{equation}
The $\alpha_{j}$, $\beta_{j}$ are given by
\begin{equation}
\alpha_{j}=\frac{\sqrt{3}}{4}\left[II_{M_{2j}}(\nu,\nu)-II_{M_{3j}}(\nu,\nu)\right],\quad\beta_{j}=\frac{3}{4}II_{M_{1j}}(\nu,\nu)\label{eq:104}
\end{equation}
The fields $\nu$ correspond to the interface of the indicated second
fundamental form.
\end{cor}

\begin{rem}
Equation (\ref{eq:103}) can be based on phase 2 or 3 instead of phase
1. In this case equations (\ref{eq:104}) must be properly modified.
The utility of this expression lies in the independence of the involved
variation components.
\end{rem}

\begin{proof}
For brevity we write $II_{pj}$ in place of $II_{M_{pj}}(\nu_{pj},\nu_{pj})$.
Considering a particular triple junction $T_{j}$, and dropping the
corresponding index, i.e. we write $II_{p}$ for $II_{pj}$ and $f_{p}$
for $f_{pj}$, we calculate the sums $\sum_{p}f_{pj}h_{pj}II_{M_{pj}}(\nu,\nu)$
on the third row of (\ref{eq:A-36}). By the matching conditions (\ref{eq:3-78})
on the spine we have 
\[
\begin{alignedat}{1}\sum_{i=1}^{3}f_{i}h_{i}II_{i} & =II_{1}f_{1}+II_{2}\left(-\tfrac{1}{2}f_{1}+\tfrac{\sqrt{3}}{2}h_{1}\right)\left(-\tfrac{\sqrt{3}}{2}f_{1}-\tfrac{1}{2}h_{1}\right)\\
 & +II_{3}\left(-\tfrac{1}{2}f_{1}-\tfrac{\sqrt{3}}{2}h_{1}\right)\left(\tfrac{\sqrt{3}}{2}f_{1}-\tfrac{1}{2}h_{1}\right)
\end{alignedat}
\]
and performing operations we obtain
\[
\sum_{i=1}^{3}f_{i}h_{i}II_{i}=\tfrac{\sqrt{3}}{4}(II_{2}-II_{3})(f_{1}^{2}-h_{1}^{2})+\left(II_{1}-\tfrac{1}{2}II_{2}-\tfrac{1}{2}II_{3}\right)f_{1}h_{1}
\]
Setting $\alpha=\frac{\sqrt{3}}{4}(II_{2}-II_{3})$, $\beta=\frac{1}{4}\left(2II_{1}-II_{2}-II_{3}\right)$,
and summing over all triple junctions in the system we obtain (\ref{eq:103}).
There remains to be proved the second of (\ref{eq:104}).

By the definition of mean curvature vector (\ref{eq:def-H}), we have
on each leaf
\[
H\cdot N=\left\langle D_{\nu}\nu,N\right\rangle +\left\langle D_{\tau}\tau,N\right\rangle 
\]
By minimality $H\cdot N=\kappa=\textrm{const.}$, hence
\[
II(\nu,\nu)=\left\langle D_{\nu}\nu,N\right\rangle =\kappa-\left\langle D_{\tau}\tau,N\right\rangle .
\]
Summation over the leaves of a triple junction, in view of (\ref{eq:mean-curv-id})
and (\ref{eq:3-73-2}), gives
\begin{equation}
II_{1}+II_{2}+II_{3}=0.\label{eq:A-33}
\end{equation}
Using this equality in $\beta=\frac{1}{4}\left(2II_{1}-II_{2}-II_{3}\right)$
proves the second of (\ref{eq:104}), and with this the proof is complete.
\end{proof}

\section{\label{sec:App-TJ-R3}Application to Triple Junction Partitioning
Problems in $\mathbb{R}^{3}$}

We apply the formula of second variation of area to prove the instability
of certain disconnected three-phase triple junction partitioning problems,
and demonstrate that the method used for treating disconnectedness
in two-phase partitionings has limited applicability to triple junction
partitioning problems.

Following the method of two-phase partitioning \cite{key-7,key-32},
we consider variations with constant normal component on each leaf.
Variations normal to all leaves of a triple junction, other than the
trivial, are not possible, for in that case the variation field $w$
would be normal to three linearly independent vectors, viz. $\tau$,
the tangent of the spine, and the two tangent fields $\nu_{1}$, $\nu_{2}$
which are normal to the spine. Thus, non-trivial tangential variations
$u$ are inevitable, at least in a neighborhood of each spine. For
simplicity we take $\gamma_{1}=\gamma_{2}=\gamma_{3}=1$. We dropped
the second index on the $\gamma$'s, as they do not depend on the
spine (see (\ref{eq:conn-not})\-). Since the variation must preserve
the volume of the phases, for the triple junction system of Figure
\ref{fig:discon-d-tj} we have by (\ref{eq:var-omega}) with $w\cdot N=v\cdot N=f$
on each leaf,
\begin{equation}
f_{1}A_{1}-f_{31}A_{31}-f_{32}A_{32}=0\label{eq:A-38}
\end{equation}
\begin{equation}
-f_{1}A_{1}+f_{21}A_{21}+f_{22}A_{22}=0\label{eq:A-39}
\end{equation}
where $A_{pj}=|M_{pj}|>0$. By (\ref{eq:3-78}) we obtain
\begin{equation}
A_{21}\tilde{h}_{1}+A_{22}\tilde{h}_{2}=\tfrac{1}{\sqrt{3}}\left(2A_{1}+A_{21}+A_{22}\right)\tilde{f}\label{eq:A-43}
\end{equation}
\begin{equation}
A_{31}\tilde{h}_{1}+A_{32}\tilde{h}_{2}=-\tfrac{1}{\sqrt{3}}\left(2A_{1}+A_{31}+A_{32}\right)\tilde{f}.\label{eq:A-44}
\end{equation}
where $\tilde{h}_{j}=h_{1j}$ and $\tilde{f}=f_{1}$. Thus the last
term of (\ref{eq:103}), on setting $\tilde{f}=1$, becomes
\[
\left(1-\tilde{h}_{1}^{2}\right)\int_{S_{1}}\alpha+2\tilde{h}_{1}\int_{S_{1}}\beta+\left(1-\tilde{h}_{2}^{2}\right)\int_{S_{2}}\alpha+2\tilde{h}_{2}\int_{S_{2}}\beta
\]
For a stable partition, $\delta^{2}A^{\star}(T)>0$ for nontrivial
variations. For constant variations this condition is
\[
\begin{split}\left(1-\tilde{h}_{1}^{2}\right)\int_{S_{1}}\alpha+2\tilde{h}_{1}\int_{S_{1}}\beta+\left(1-\tilde{h}_{2}^{2}\right)\int_{S_{2}}\alpha+2\tilde{h}_{2}\int_{S_{2}}\beta>\\
\sum_{p,j}f_{pj}^{2}\int_{M_{pj}}|B_{M_{pj}}|^{2}+\sum_{p,j}f_{pj}^{2}\int_{\partial M_{pj}\cap\varSigma}II_{\varSigma}(N_{pj},N_{pj})
\end{split}
\]
For convex $\Omega$ the expression on the right side is non-negative,
and if we prove that the expression on the left side is non-positive,
i.e.
\[
\left(1-\tilde{h}_{1}^{2}\right)\int_{S_{1}}\alpha+2\tilde{h}_{1}\int_{S_{1}}\beta+\left(1-\tilde{h}_{2}^{2}\right)\int_{S_{2}}\alpha+2\tilde{h}_{2}\int_{S_{2}}\beta\leqslant0
\]
we get a contradiction, and this would prove that the partitioning
is not stable. From what we show in the sequel (see Section \ref{sec:Apps-2D})
it turns out that this condition does not hold in general, and thus
the methods of two-phase partitioning are in general not applicable
to triple junction partitioning problems.

However, this method can be used to prove instability in certain cases.
For example, in the disconnected partitioning of Figure \ref{fig:discon-d-tj},
assuming that $M_{1}$ is flat, i.e. $II_{1}=0$, we have $\beta=0$.
Further, we assume $A_{21}=A_{22}$, $A_{31}=A_{32}$ and $II_{2}\geqslant0$
for both spines. In this case the system (\ref{eq:A-43})-(\ref{eq:A-44})
has no solution, except when $f_{1}=0$, $\tilde{h}_{2}=-\tilde{h}_{1}$,
and then the above condition reduces to
\[
-\tilde{h}_{1}^{2}\int_{S_{1}}\alpha-\tilde{h}_{2}^{2}\int_{S_{2}}\alpha\leqslant0
\]
which is true by the hypothesis $II_{2}\geqslant0$. We have proved
the following Proposition.
\begin{prop}
\label{prop:Gen-instab-res}Let $\Omega$ be a convex domain in $\mathbb{R}^{3}$
and $T=(T_{j})_{j=1}^{2}$, $T_{j}=(M_{pj})_{p=1}^{3}$, a minimal
disconnected three-phase partition of $\Omega$ by a system of two
$C^{2}$ triple junctions as in Figure \ref{fig:discon-d-tj}, with
volume constraints. We assume that $\Sigma=\partial\Omega$ is $C^{2}$
in a neighborhood of $\Sigma\cap\overline{T}$. Further we assume
that $M_{1}$ is flat, the areas of the leaves $A_{pj}=|M_{pj}|$
satisfy the condition
\[
\left|\begin{array}{cc}
A_{21} & A_{22}\\
A_{31} & A_{32}
\end{array}\right|=0
\]
and $II_{2j}(\nu,\nu)\geqslant0$ for both spines $j=1,2$. Then the
disconnected triple junction partition $T$ is unstable.
\end{prop}

\section[Spectral Analysis]{\label{sec:Spectr-Anal}Spectral Analysis of the 2nd Variation Form}

To keep the length of formulas to a minimum and focus on the essence
of the argument, we present the details for the configuration of Figure
\ref{fig:discon-d-tj}, and adopt the assumption $\gamma_{1}=\gamma_{2}=\gamma_{3}=1$
from this point on.

Let $T$ be a system of triple junctions of a three phase partitioning
problem in $\Omega$, which is assumed minimal, i.e. $\delta A(T)=0$.
When $T$ is stable, it is a local minimizer of the area functional,
so we aim at studying the conditions under which $T$ is stable. To
this purpose we will use a spectral analysis method for the bilinear
form expressing the second variation of area of $T$,
\begin{equation}
\begin{gathered}J(\mathbf{f})=\sum_{p,j}\int_{M_{pj}}\left(|\nabla^{M_{pj}}f_{pj}|^{2}-|B_{M_{pj}}|^{2}f_{pj}^{2}\right)-\sum_{p,j}\int_{\partial M_{pj}\cap\varSigma}\sigma_{pj}f_{pj}^{2}\\
+\sum_{j}\int_{S_{j}}\left[\alpha_{j}\left(f_{1j}^{2}-h_{1j}^{2}\right)+2\beta_{j}f_{1j}h_{1j}\right]
\end{gathered}
\label{eq:J}
\end{equation}
where $\sigma_{pj}=II_{\varSigma}(N_{pj},N_{pj})$ and $\mathbf{f}=(f_{1},f_{21},f_{31},f_{22},f_{32})$.
As $f_{11}\equiv f_{12}$ we write simply $f_{1}$. For brevity we
will write $\nabla^{M}f$ in place of $\mathrm{grad}_{M}\,f$. Although
$J$ and $\delta^{2}A^{\star}(T)$ are identical expressions, their
meaning is different: in the context of spectral analysis $T$ is
a \emph{fixed} system of manifolds with boundary and $J$ is a nonlinear
functional on a properly defined functional space on $T$ containing
the admissible variations of $T$. As a consequence the functions
of this space satisfy the conditions of volume constancy
\begin{equation}
\begin{cases}
-\int_{M_{1}}f_{1}+\int_{M_{21}}f_{21}+\int_{M_{22}}f_{22}=0\\
-\int_{M_{1}}f_{1}+\int_{M_{31}}f_{31}+\int_{M_{32}}f_{32}=0
\end{cases}\label{eq:vol-const}
\end{equation}
and the normalization condition
\begin{equation}
\sum_{p,j}\int_{M_{pj}}f_{pj}^{2}=\int_{M_{1}}f_{1}^{2}+\int_{M_{21}}f_{21}^{2}+\int_{M_{31}}f_{31}^{2}+\int_{M_{22}}f_{22}^{2}+\int_{M_{32}}f_{32}^{2}=1\label{eq:normal}
\end{equation}
in addition to the compatibility conditions on the spines (see first
of (\ref{eq:3-76})\-),
\begin{equation}
\begin{cases}
f_{1}+f_{21}+f_{31}=0 & \textrm{on }S_{1}\\
f_{1}+f_{22}+f_{32}=0 & \textrm{on }S_{2}
\end{cases}\label{eq:comp-cond-sp}
\end{equation}
For convenience, we introduce Lagrange multipliers $\lambda_{2},\lambda_{3}$,
and the corresponding functional
\begin{equation}
\begin{gathered}J^{\star}(\mathbf{f};\mu,\lambda_{2},\lambda_{3})=J(\mathbf{f})-\mu\left(\sum_{p,j}\int_{M_{pj}}f_{pj}^{2}-1\right)\\
-\lambda_{2}\left(-\int_{M_{1}}f_{1}+\int_{M_{21}}f_{21}+\int_{M_{22}}f_{22}\right)\\
-\lambda_{3}\left(-\int_{M_{1}}f_{1}+\int_{M_{31}}f_{31}+\int_{M_{32}}f_{32}\right)
\end{gathered}
\label{eq:J-1}
\end{equation}
and we are interested in the critical points of $J^{\star}$.
\begin{prop}
\label{prop:eigv-prob}A necessary and sufficient condition for a
$C^{2}$ function $\mathbf{f}=(f_{1},f_{21},f_{31},f_{22},f_{32})$
on $T$, which satisfies the compatibility conditions (\ref{eq:comp-cond-sp})
on the spines, to be a critical point of $J^{\star}$, or equivalently
of $J$ with the conditions (\ref{eq:vol-const}) and (\ref{eq:normal}),
is that it satisfies the following linear inhomogeneous system of
PDE
\begin{equation}
\Delta_{M_{pj}}f_{pj}+\left(\mu+|B_{M_{pj}}|^{2}\right)f_{pj}=-\tfrac{1}{2}\lambda_{pj}\label{eq:eigv-PDE}
\end{equation}
with Neumann-type boundary conditions:
\begin{equation}
D_{\nu_{pj}}f_{pj}=\sigma_{pj}f_{pj}\quad\textrm{on}\,\partial M_{pj}\cap\Sigma\label{eq:eigv-BC}
\end{equation}
\begin{equation}
\alpha_{j}f_{1j}+\beta_{j}h_{1j}=-D_{\nu_{1j}}f_{1j}+\tfrac{1}{2}\left(D_{\nu_{2j}}f_{2j}+D_{\nu_{3j}}f_{3j}\right)\label{eq:eigv-sp1}
\end{equation}
\begin{equation}
\beta_{j}f_{1j}-\alpha_{j}h_{1j}=-\tfrac{\sqrt{3}}{2}\left(D_{\nu_{2j}}f_{2j}-D_{\nu_{3j}}f_{3j}\right)\label{eq:eigv-sp2}
\end{equation}
\end{prop}

\begin{rem}
In (\ref{eq:eigv-PDE}) $\Delta_{M}$ is the Laplace-Beltrami operator
on $M$ defined by
\[
\Delta_{M}f=\mathrm{div}_{M}(\nabla^{M}f)=g^{-1/2}(g^{1/2}g^{ij}f_{,j})_{,i}
\]
in a local coordinate system $q^{1},\cdots,q^{n-1}$, where $g=\det\left[g_{ij}\right]$,
$g_{ij}$ is the metric tensor and the comma operator denotes partial
derivative in the respective coordinate, i.e. $f_{,i}=\frac{\partial f}{\partial q^{i}}=D_{E_{i}}f$.
As $M$ is fixed, $g_{ij}$ is fixed and (\ref{eq:eigv-PDE}) is a
linear equation.
\end{rem}

\begin{rem}
The PDE's (\ref{eq:eigv-PDE}) and the boundary conditions (\ref{eq:eigv-BC})
are independent of the particular problem, provided that the $\lambda$'s
have been defined properly. The boundary conditions (\ref{eq:eigv-sp1})
and (\ref{eq:eigv-sp2}), as well as the definition of the $\lambda$'s
are problem specific. For the problem at hand the particular form
of the PDE's is
\begin{equation}
\begin{cases}
\Delta_{M_{21}}f_{21}+\left(\mu+|B_{M_{21}}|^{2}\right)f_{21}=-\tfrac{1}{2}\lambda_{2}\\
\Delta_{M_{21}}f_{21}+\left(\mu+|B_{M_{21}}|^{2}\right)f_{21}=-\tfrac{1}{2}\lambda_{2}\\
\Delta_{M_{22}}f_{22}+\left(\mu+|B_{M_{22}}|^{2}\right)f_{22}=-\tfrac{1}{2}\lambda_{3}\\
\Delta_{M_{22}}f_{22}+\left(\mu+|B_{M_{22}}|^{2}\right)f_{22}=-\tfrac{1}{2}\lambda_{3}\\
\Delta_{M_{1}}f_{1}+\left(\mu+|B_{M_{1}}|^{2}\right)f_{1}=\tfrac{1}{2}(\lambda_{2}+\lambda_{3})
\end{cases}\label{eq:eigv-PDE-exp}
\end{equation}
\end{rem}

\begin{proof}
The first variation of $J^{\star}$ is given by
\[
\begin{alignedat}{1}\delta J^{\star}(\mathbf{f})\boldsymbol{\phi} & =\left.\frac{d}{dt}J^{\star}(\mathbf{f}+t\boldsymbol{\phi})\right|_{t=0}\\
 & =2\sum_{p,j}\int_{M_{pj}}\left(\nabla^{M_{pj}}f_{pj}\cdot\nabla^{M_{pj}}\phi_{pj}-|B_{M_{pj}}|^{2}f_{pj}\phi_{pj}\right)\\
 & -2\sum_{p,j}\int_{\partial M_{pj}}\sigma_{pj}f_{pj}\phi_{pj}-2\mu\sum_{p,j}\int_{M_{pj}}f_{pj}\phi_{pj}\\
 & +2\sum_{j}\int_{S_{j}}\left[\alpha_{j}\left(f_{1j}\phi_{1j}-h_{1j}\psi_{1j}\right)+\beta_{j}\left(\phi_{1j}h_{1j}+f_{1j}\psi_{1j}\right)\right]\\
 & -\lambda_{2}\left(-\int_{M_{1}}\phi_{1}+\int_{M_{21}}\phi_{21}+\int_{M_{22}}\phi_{22}\right)\\
 & -\lambda_{3}\left(-\int_{M_{1}}\phi_{1}+\int_{M_{31}}\phi_{31}+\int_{M_{32}}\phi_{32}\right)
\end{alignedat}
\]
with
\[
\psi_{1j}=\delta h_{1j}=\delta\left[\tfrac{1}{\sqrt{3}}\left(f_{1j}+2f_{2j}\right)\right]=\tfrac{1}{\sqrt{3}}\left(\phi_{1j}+2\phi_{2j}\right)
\]
The second equality is obtained by the first of (\ref{eq:3-78}).
By Green's formula for manifolds, the term on the second row is written
in the form
\[
-2\sum_{p,j}\int_{M_{pj}}\phi_{pj}\left(\Delta_{M_{pj}}f_{pj}+|B_{M_{pj}}|^{2}f_{pj}\right)+2\sum_{p,j}\int_{\partial M_{pj}}\phi_{pj}D_{\nu_{pj}}f_{pj}
\]
and splitting the integrals over $\partial M_{pj}$ into boundary
and spine parts
\[
\int_{\partial M_{pj}}\phi_{pj}D_{\nu_{pj}}f_{pj}=\int_{\partial M_{pj}\cap\Sigma}\phi_{pj}D_{\nu_{pj}}f_{pj}+\int_{S_{j}}\phi_{pj}D_{\nu_{pj}}f_{pj}
\]
we obtain
\[
\begin{alignedat}{1}\delta J^{\star}(\mathbf{f})\boldsymbol{\phi} & =-2\sum_{p,j}\int_{M_{pj}}\phi_{pj}\left[\Delta_{M_{pj}}f_{pj}+\left(\mu+|B_{M_{pj}}|^{2}\right)f_{pj}\right]\\
 & +2\sum_{p,j}\int_{\partial M_{pj}\cap\Sigma}\phi_{pj}\left(D_{\nu_{pj}}f_{pj}-\sigma_{pj}f_{pj}\right)\\
 & +2\sum_{p,j}\int_{S_{j}}\phi_{pj}D_{\nu_{pj}}f_{pj}\\
 & +2\sum_{j}\int_{S_{j}}\left[\alpha_{j}\left(f_{1j}\phi_{1j}-h_{1j}\psi_{1j}\right)+\beta_{j}\left(\phi_{1j}h_{1j}+f_{1j}\psi_{1j}\right)\right]\\
 & -\sum_{p,j}\lambda_{pj}\int_{M_{pj}}\phi_{pj}
\end{alignedat}
\]
We have set $\lambda_{11}=-(\lambda_{2}+\lambda_{3})$, $\lambda_{21}=\lambda_{22}=\lambda_{2}$
and $\lambda_{31}=\lambda_{32}=\lambda_{3}$. Using compactly supported
variations vanishing on all $\partial M_{pj}\cap\Sigma$ and $S_{j}$
we obtain \textbf{(}\ref{eq:eigv-PDE}). Variations $\boldsymbol{\phi}$
concentrated on $\partial M_{pj}\cap\Sigma$ give (\ref{eq:eigv-BC}).
The remaining terms are integrals over spines:
\[
\begin{alignedat}{1}\delta J^{\star}(\mathbf{f})\boldsymbol{\phi} & =2\sum_{p,j}\int_{S_{j}}\phi_{pj}D_{\nu_{pj}}f_{pj}\\
 & +2\sum_{j}\int_{S_{j}}\left[\alpha_{j}\left(f_{1j}\phi_{1j}-h_{1j}\psi_{1j}\right)+\beta_{j}\left(\phi_{1j}h_{1j}+f_{1j}\psi_{1j}\right)\right]
\end{alignedat}
\]
Using the identities (\ref{eq:3-76}) to express the $\phi_{3j}$
in terms of independent quantities,
\[
\phi_{3j}=\delta f_{3j}=-\delta f_{1j}-\delta f_{2j}=-\phi_{1j}-\phi_{2j},
\]
expanding out and collecting similar terms,
\[
\begin{alignedat}{1}\delta J^{\star}(\mathbf{f}) & =2\sum_{j}\int_{S_{j}}\phi_{1j}D_{\nu_{1j}}f_{1j}+\phi_{2j}D_{\nu_{2j}}f_{2j}+\phi_{3j}D_{\nu_{3j}}f_{3j}\\
 & +2\sum_{j}\int_{S_{j}}\left[\alpha_{j}\left(f_{1j}\phi_{1j}-h_{1j}\psi_{1j}\right)+\beta_{j}\left(\phi_{1j}h_{1j}+f_{1j}\psi_{1j}\right)\right]\\
 & =2\sum_{j}\int_{S_{j}}\phi_{1j}\left(D_{\nu_{1j}}f_{1j}-D_{\nu_{3j}}f_{3j}\right)+\phi_{2j}\left(D_{\nu_{2j}}f_{2j}-D_{\nu_{3j}}f_{3j}\right)\\
 & +2\sum_{j}\int_{S_{j}}\phi_{1j}\left(\alpha_{j}f_{1j}+\beta_{j}h_{1j}\right)+2\sum_{j}\int_{S_{j}}\psi_{1j}\left(\beta_{j}f_{1j}-\alpha_{j}h_{1j}\right)
\end{alignedat}
\]
Substituting for $\phi_{2j}$ by $\phi_{2j}=\frac{\sqrt{3}}{2}\psi_{1j}-\frac{1}{2}\phi_{1j}$,
and performing operations gives
\[
\begin{alignedat}{1}\delta J^{\star}(\mathbf{f}) & =2\sum_{j}\int_{S_{j}}\phi_{1j}\left[D_{\nu_{1j}}f_{1j}-\tfrac{1}{2}D_{\nu_{2j}}f_{2j}-\tfrac{1}{2}D_{\nu_{3j}}f_{3j}+\alpha_{j}f_{1j}+\beta_{j}h_{1j}\right]\\
 & +2\sum_{j}\int_{S_{j}}\psi_{1j}\left(\tfrac{2}{\sqrt{3}}D_{\nu_{2j}}f_{2j}-\tfrac{2}{\sqrt{3}}D_{\nu_{3j}}f_{3j}+\beta_{j}f_{1j}-\alpha_{j}h_{1j}\right)
\end{alignedat}
\]
Using variations concentrated on each spine we obtain (\ref{eq:eigv-sp1})
and (\ref{eq:eigv-sp2}). The converse is immediate.
\end{proof}
In the following we present two propositions that are necessary for
the study of stability of triple junction partitionings. The next
proposition states that a partitioning problem is unstable if there
is an eigenvalue $\mu<0$ of the system (\ref{eq:eigv-PDE})-(\ref{eq:eigv-sp2}).
\begin{prop}
\label{prop:Categ-a}Let $T$ be a system of $C^{2}$ triple junctions
of a minimal three phase partitioning problem in $\Omega$, and $\mathbf{f}$
an eigenfunction of problem (\ref{eq:eigv-PDE})-(\ref{eq:eigv-sp2})
with corresponding eigenvalue $\mu$. Then
\begin{equation}
J(\mathbf{f})=\mu.\label{eq:J-min}
\end{equation}
In particular, if $\mu<0$, $T$ is unstable.
\end{prop}

\begin{rem}
Proposition \ref{prop:Categ-a} implies that, with a negative eigenvalue
at hand, no lower bound is needed to be known in advance for the functional
$J$, in order to conclude that a minimal partitioning is unstable.
\end{rem}

\begin{proof}
Multiplication of (\ref{eq:eigv-PDE}) by $f_{pj}$, integration over
$M_{pj}$ and summation over all leaves of all triple junctions, gives
in view of (\ref{eq:vol-const}) and (\ref{eq:normal})
\[
\sum_{p,j}\int_{M_{pj}}\left(f_{pj}\Delta_{M_{pj}}f_{pj}+|B_{M_{pj}}|^{2}f_{pj}^{2}\right)+\mu=0
\]
Application of Green's formula gives
\[
\sum_{p,j}\int_{M_{pj}}\left(|\nabla^{M_{pj}}f_{pj}|^{2}-|B_{M_{pj}}|^{2}f_{pj}^{2}\right)-\sum_{p,j}\int_{\partial M_{pj}}f_{pj}D_{\nu_{pj}}f_{pj}=\mu
\]
On breaking the integrals over $\partial M_{pj}$ into boundary and
spine parts and using the boundary condition (\ref{eq:eigv-BC}),
we obtain
\[
\begin{gathered}\sum_{p,j}\int_{M_{pj}}\left(|\nabla^{M_{pj}}f_{pj}|^{2}-|B_{M_{pj}}|^{2}f_{pj}^{2}\right)-\sum_{p,j}\int_{\partial M_{pj}\cap\Sigma}\sigma_{pj}f_{pj}^{2}\\
-\sum_{p,j}\int_{S_{j}}f_{pj}D_{\nu_{pj}}f_{pj}=\mu
\end{gathered}
\]
Furthermore, on each spine we have
\[
\begin{gathered}\alpha_{j}\left(f_{1j}^{2}-h_{1j}^{2}\right)+2\beta_{j}f_{1j}h_{1j}=f_{1j}(\alpha_{j}f_{1j}+\beta_{j}h_{1j})+h_{1j}(\beta_{j}f_{1j}-\alpha_{j}h_{1j})=\\
f_{1j}\left[-D_{\nu_{1j}}f_{1j}+\tfrac{1}{2}\left(D_{\nu_{2j}}f_{2j}+D_{\nu_{3j}}f_{3j}\right)\right]-\tfrac{\sqrt{3}}{2}h_{1j}\left(D_{\nu_{2j}}f_{2j}-D_{\nu_{3j}}f_{3j}\right)=\\
\sum_{p}f_{pj}D_{\nu_{pj}}f_{pj}
\end{gathered}
\]
after using (\ref{eq:3-78}) and performing straight-forward operations.
By the definition of $J$ (\ref{eq:J}) we obtain $J(\mathbf{f})=\mu$.
The second assertion follows immediately from this.
\end{proof}
Proposition \ref{prop:Categ-a} can be used to prove instability.
The method we follow to establish the stability of a specific partitioning
problem, is to prove that the minimal eigenvalue of $J$, or equivalently
of the boundary value problem (\ref{eq:eigv-PDE})-(\ref{eq:eigv-sp2}),
is positive. We provide a justification of this method. The difficulty
is that the boundary integral $\int_{\partial M}f^{2}$ cannot be
bounded above by $\int_{M}f^{2}$. However, if $f\in W^{1,2}(M)\equiv H^{1}(M)$,
the boundary trace embedding theorem (\cite{key-103} §§ 5.34-5.37,
pp 163-166; \cite{key-104} § 8, pp 120-132) and a suitable interpolation
estimate (see Lemma \ref{lem:interp-est} below) yield the coercivity
of $J$. From this by a well-known theorem (see proof of Proposition
\ref{prop:Categ-b}) we obtain the existence of a minimal eigenvalue
for $J$.

The standard notation for Sobolev spaces is used: $|u|_{L^{2}(M)}=(\int_{M}u^{2})^{1/2}$,
$|u|_{H^{1}(M)}=(|u|_{L^{2}(M)}^{2}+|\nabla^{M}u|_{L^{2}(M)}^{2})^{1/2}$,
$|u|_{L^{2}(\partial M)}=(\int_{\partial M}u^{2})^{1/2}$ are the
standard norms of $L^{2}(M)$, $H^{1}(M)$ and $L^{2}(\partial M)$.
Let $T$ be a partitioning of $\Omega$ by a system of triple junctions.
Sobolev spaces on $T$ are defined as follows: 
\[
H^{1}(T)=\prod_{j=1}^{r}\prod_{p=1}^{3}H^{1}(M_{pj})
\]
It is understood that each distinct $M_{pj}$ participates in the
product only once. The $L^{2}$-norm is 
\[
|\mathbf{f}|_{L^{2}(T)}=\left(\sum_{p,j}|f_{pj}|_{L^{2}(M_{pj})}^{2}\right)^{1/2}
\]
and the $H^{1}(T)$-norm is defined analogously. Further, we define
the functionals
\[
\varphi_{p}(\mathbf{f})=\sum_{j=1}^{r}\sum_{q\neq p}\varepsilon_{pqj}\int_{M_{qj}}f_{qj}\quad(p,q=1,2,3)
\]
where
\[
\varepsilon_{pqj}=\begin{cases}
+1, & N_{qj}\:\textrm{outward}\:\textrm{to}\:\Omega_{pj}\\
-1, & N_{qj}\:\textrm{inward}\:\textrm{to}\:\Omega_{pj}
\end{cases}
\]
Clearly $\sum_{p}\varphi_{p}(\mathbf{f})=0$ and the volume constraints
can be expressed by means of these functionals.
\begin{example}
For the triple junction system of Figure \ref{fig:discon-d-tj} the
$\varphi$-functionals are
\[
\begin{alignedat}{1}\varphi_{1}(\mathbf{f})= & \int_{M_{21}}f_{21}+\int_{M_{22}}f_{22}-\int_{M_{31}}f_{31}-\int_{M_{32}}f_{32}\\
\varphi_{2}(\mathbf{f})= & -\int_{M_{1}}f_{1}+\int_{M_{31}}f_{31}+\int_{M_{32}}f_{32}\\
\varphi_{3}(\mathbf{f})= & -\int_{M_{1}}f_{1}+\int_{M_{21}}f_{21}+\int_{M_{22}}f_{22}
\end{alignedat}
\]
The volume constraints are given by $\varphi_{2}(\mathbf{f})=0$,
$\varphi_{3}(\mathbf{f})=0$.
\end{example}

The following estimate is easily established for $\mathbf{f}\in L^{2}(T)$:
\begin{equation}
|\varphi_{p}(\mathbf{f})|\leqslant c_{0}|\mathbf{f}|_{L^{2}(T)},\quad p=1,2,3.\label{eq:105}
\end{equation}
\begin{lem}
\label{lem:interp-est}Let $M$ be a bounded $C^{2}$ submanifold
of $\mathbb{R}^{n}$ with boundary. Then for every $\epsilon>0$ there
is a constant $c_{\epsilon}$ such that for any $u\in H^{1}(M)$
\begin{equation}
|u|_{L^{2}(\partial M)}\leqslant\epsilon|u|_{H^{1}(M)}+c_{\epsilon}|u|_{L^{2}(M)}\label{eq:BT-new}
\end{equation}
\end{lem}

For the proof see \cite{key-32}.

After this preparation we can prove the following result, which is
the basis of our method for establishing stability.
\begin{prop}
\label{prop:Categ-b}Let $T$ be a minimal three phase partitioning
of $\Omega\subset\mathbb{R}^{3}$ by a system of triple junctions.
Then for any $\mathbf{f}=(f_{pj})\in H^{1}(T)$ satisfying the compatibility
conditions on the spines and (\ref{eq:vol-const}), (\ref{eq:normal})
we have 
\begin{equation}
J(\mathbf{f})\geqslant\mu_{1}\label{eq:J-min-eigv}
\end{equation}
where $\mu_{1}$ is the smallest eigenvalue of problem (\ref{eq:eigv-PDE})-(\ref{eq:eigv-sp2}).
In particular, if $\mu_{1}>0$, $T$ is stable.
\end{prop}

\begin{proof}
We will prove that the conditions of Theorem 1.2 in \cite{key-105}
are satisfied. Let 
\[
X=\{u\in H^{1}(T):|u|_{L^{2}(T)}\leqslant1,\,\sum_{p}u_{pj}=0\,\textrm{on }S_{j},\,\varphi_{2}(u)=\varphi_{3}(u)=0\}.
\]
By the continuity of the $L^{2}(T)$-norm, the functionals $\varphi_{p}:H^{1}(T)\to\mathbb{R}$,
and the mappings $u\mapsto\left.\sum_{p}u_{pj}\right|_{S_{j}}$, it
follows that $X$ is a closed subset of $H^{1}(T)$. The convexity
of $X$ is clear. Hence $X$ is a weakly closed subset of $H^{1}(T)$.
We prove the coercivity of $J$ on $X$. By continuity there are non-negative
constants $\sigma_{0}$, $b_{0}$, $\alpha_{0}$ such that
\[
\sigma_{pj}\leqslant\sigma_{0},\;|B_{M_{pj}}|^{2}\leqslant b_{0},\;|\alpha_{j}|\leqslant\tfrac{1}{2}\alpha_{0},\;|\beta_{j}|\leqslant\tfrac{1}{2}\alpha_{0}
\]
for all $p,j$. By the inequalities
\[
\alpha_{j}\left(f_{1j}^{2}-h_{1j}^{2}\right)\geqslant-\tfrac{1}{2}\alpha_{0}\left(f_{1j}^{2}+h_{1j}^{2}\right),\quad2\beta_{j}f_{1j}h_{1j}\geqslant-\alpha_{0}|f_{1j}h_{1j}|
\]
which are easily established, we obtain
\[
\alpha_{j}\left(f_{1j}^{2}-h_{1j}^{2}\right)+2\beta_{j}f_{1j}h_{1j}\geqslant-\alpha_{0}\left(f_{1j}^{2}+h_{1j}^{2}\right)
\]
and using the first of (\ref{eq:3-78}),
\[
\alpha_{j}\left(f_{1j}^{2}-h_{1j}^{2}\right)+2\beta_{j}f_{1j}h_{1j}\geqslant-2\alpha_{0}\left(f_{1j}^{2}+f_{2j}^{2}\right)
\]
Thus by (\ref{eq:J}), with possibly redefined values of $b_{0}$,
$\sigma_{0}$, we obtain the estimate
\[
J(\mathbf{f})\geqslant\sum_{p,j}|f_{pj}|_{H^{1}(M_{pj})}^{2}-b_{0}\sum_{p,j}|f_{pj}|_{L^{2}(M_{pj})}^{2}-\sigma_{0}\sum_{p,j}\int|f_{pj}|_{L^{2}(\partial M_{pj})}^{2}
\]
By the interpolation inequality (\ref{eq:BT-new}),
\[
J(\mathbf{f})\geqslant(1-\epsilon\sigma_{0})\sum_{p,j}|f_{pj}|_{H^{1}(M_{pj})}^{2}-(b_{0}+c_{\epsilon}\sigma_{0})\sum_{p,j}|f_{pj}|_{L^{2}(M_{pj})}^{2}
\]
which for sufficiently small value of $\epsilon>0$ proves the coercivity
of $J$. The sequential weakly lower semicontinuity of $J$ follows
from the sequential weakly lower semicontinuity of the norm of $H^{1}(T)$
and the compactness of the embedding $H^{1}(M)\hookrightarrow L^{2}(M)$.
Thus the conditions of Theorem 1.2 in \cite{key-105} are satisfied,
and from this we conclude that $J$ attains its infimum in $X$. The
position of the infimum is a critical point of $J^{\star}$ and, as
it was shown in Proposition \ref{prop:Categ-a}, it is a solution
of equation (\ref{eq:eigv-PDE}) with BC (\ref{eq:eigv-BC}). The
inequality (\ref{eq:J-min-eigv}) follows immediately form this.
\end{proof}

\section{\label{sec:Apps-2D}Application to Disconnected 2-Dimensional Partitioning
Problems}

We prove the existence of stable disconnected partitionings in 2 dimensions
by example.

\begin{figure}
\includegraphics[scale=0.75]{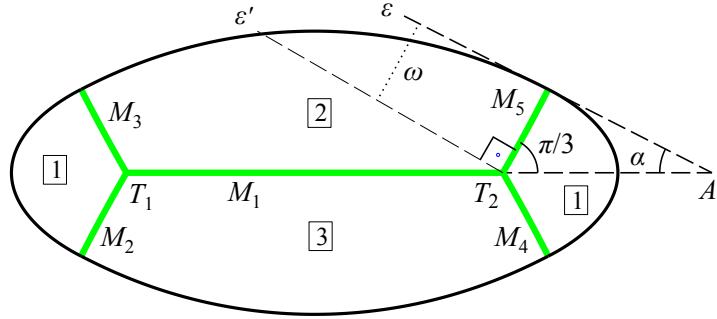}

\caption{\label{fig:DTJ}Disconnected three phase partitioning consisting of
two triple junctions. Boxed numbers indicate phases. $T_{1}$ and
$T_{2}$ are the triple junctions. The leaves of the triple junctions
$M_{2},\cdots,M_{5}$ are circular arcs of the same curvature $\kappa$
(absolute value) while $M_{1}$ is flat. By minimality the tangent
to $M_{5}$ at $T_{2}$ is at angle $\pi/3$ with the line $T_{1}A$.
$A\varepsilon$ is the tangent to $\Sigma=\partial\Omega$ at $M_{5}\cap\Sigma$
and $T_{2}\varepsilon^{\prime}$ is the normal to the tangent of $M_{5}$
at $T_{2}$. The intersection $C$ (not shown in the Figure) of $A\varepsilon$
and $T_{2}\varepsilon^{\prime}$ is the center of the circle with
radius $R=1/\kappa$ containing the circular arc $M_{5}$. Since $\omega=\widehat{ACT}_{2}$,
by plane geometry, $\omega=\frac{\pi}{6}-\alpha$. Analogous relations
hold for the other leaves. This Figure was produced by an Octave (MATLAB)
program. The boundary $\Sigma$ was drawn as a set of four four-degree
splines having curvature zero at their ends $M_{i}\cap\Sigma$, $i=2,\cdots,5$.
The indicated boundary is symmetric about the $x$ and $y$ axes,
but a non-symmetric boundary could also be drawn for the same triple
junction arrangement. }

\end{figure}

For two-dimensional partitions the Laplace-Beltrami operator reduces
to
\[
\Delta_{M}f=\frac{d^{2}f}{ds^{2}}
\]
where $s$ is the arc length of $M$, $M$ being any triple junction
leaf, and the integrals over $\partial M$ reduce to numbers. The
boundary condition (\ref{eq:eigv-BC}) reduces to
\begin{equation}
\frac{df}{ds}=0,\;\textrm{on}\,\partial M\cap\Sigma\label{eq:BC-2d}
\end{equation}
As we are interested in proving the existence of \emph{stable disconnected
partitionings}, we have chosen $\sigma=0$ at $\partial M\cap\Sigma$.
From part (ii) of Proposition \ref{prop:1st-var-gen} the curvature
of each leaf is constant, thus the only possibilities for $M$ are
line segments and circular arcs. Further, $|B_{M}|=\kappa=1/R$, where
$R$ is the radius of the arc or $\infty$ for line segments. We consider
disconnected partitionings having the topology of Figure \ref{fig:DTJ}.
We are using the sequential enumeration notation of Example \ref{exa:primary}
in which $\mathbf{f}=(f_{1},\cdots,f_{5})$, $T=(M_{1},\cdots,M_{5})$. 

For equation (\ref{eq:eigv-PDE}) we have the following three types
of solution, depending on the sign of $\mu+|B_{M}|^{2}$,
\[
\begin{aligned}(I) & \quad f(s)=-\frac{\lambda}{2k^{2}}+C\sin(ks)+D\cos(ks), & k^{2}=\mu+\kappa^{2}\Leftrightarrow\mu>-\kappa^{2}\\
(II) & \quad f(s)=\frac{\lambda}{2k^{2}}+Ce^{ks}+De^{-ks}, & k^{2}=-(\mu+\kappa^{2})\Leftrightarrow\mu<-\kappa^{2}\\
(III) & \quad f(s)=-\frac{\lambda}{4}s^{2}+Cs+D, & \mu=-\kappa^{2}
\end{aligned}
\]
and
\[
\lambda_{1}=-(\lambda_{2}+\lambda_{3}),\;\lambda_{4}=\lambda_{2},\;\lambda_{5}=\lambda_{3}
\]
while the $\lambda_{2},\,\lambda_{3}$ are independent variables.
The parametrizations of the $M_{i}$ (with the exception of $M_{1}$
which does not intersect $\Sigma$) are such that $M_{i}\cap\Sigma$
is obtained at $s=l_{i}$, where $l_{i}=|M_{i}|$ is the length of
$M_{i}$. From the BC (\ref{eq:BC-2d}) we obtain for the three cases
($i\neq1)$
\[
(I)\;D_{i}=C_{i}\cot(k_{i}l_{i}),\quad(II)\;D_{i}=C_{i}e^{2k_{i}l_{i}},\quad(III)\;C_{i}=-\tfrac{1}{2}\lambda_{i}l_{i}
\]
and
\[
\begin{aligned}(I) & \quad f_{i}(s)=-\frac{\lambda_{i}}{2k_{i}^{2}}+\frac{C_{i}}{\sin(k_{i}l_{i})}\cos k_{i}(s-l_{i}),\\
(II) & \quad f_{i}(s)=\frac{\lambda_{i}}{2k_{i}^{2}}+C_{i}e^{2k_{i}l_{i}}\cosh k_{i}(s-l_{i}),\\
(III) & \quad f(s)=-\frac{\lambda_{i}}{4}(s-l_{i})^{2}+\frac{\lambda_{i}l_{i}^{2}}{4}+D_{i},
\end{aligned}
\]

Since $f_{1}$ contains 2 constants while all other $f$'s only one,
we have in total 6 unknown constants, which together with $\lambda_{2}$,
and $\lambda_{3}$ make 8 unknowns. On the other hand, two volume
constraints, the BC's (\ref{eq:eigv-sp1}), (\ref{eq:eigv-sp2}),
and the compatibility condition on the two spines (see first of (\ref{eq:3-76}))
are 8 equations in total, and in this way we have a linear system
of 8 equations in 8 unknowns. The condition for existence of solutions
of this system is, as usually, obtained by setting its determinant
to 0, which gives a nonlinear equation for $k$. With a solution for
$k$ at hand, we can determine the eigenvalue $\mu$ by the last column
in the above table of possible solutions for $f$, and each eigenvector
determines an eigenfunction $f$ of problem (\ref{eq:eigv-PDE})-(\ref{eq:eigv-sp2}).

We specialize these relations by considering the case of Figure \ref{fig:DTJ}
in which the leaves $M_{2}$, $M_{3}$, $M_{4}$ and $M_{5}$ have
the same radius $R=1/\kappa$ and the same length $l$, while $M_{1}$
is flat and has length $L$. In this case
\[
II_{M_{21}}(\nu,\nu)=II_{M_{22}}(\nu,\nu)=-\frac{1}{R}=-\kappa
\]
\[
II_{M_{31}}(\nu,\nu)=II_{M_{32}}(\nu,\nu)=\kappa,\;II_{M_{1}}(\nu,\nu)=-(II_{M_{21}}+II_{M_{31}})=0
\]
and
\[
\alpha=-\frac{\sqrt{3}}{2}\kappa<0,\quad\beta=0
\]
for both spines.

We distinguish the following cases for $\mu$:

\subsection{Case I: $-\kappa^{2}<\mu<0,\;k^{2}=\mu+\kappa^{2}\,(0<\frac{k}{\kappa}<1)$}

For $i\neq1$ we have $\mu+|B_{M_{i}}|^{2}=\mu+\kappa^{2}>0$, and
$\mu+|B_{M_{i}}|^{2}=\mu<0$, so we have a solution type (I) on $M_{2}$,
$M_{3}$, $M_{4}$, $M_{5}$ and type (II) on $M_{1}$. Thus
\begin{equation}
\begin{aligned} & f_{1}(s)=-\frac{\lambda_{2}+\lambda_{3}}{2(\kappa^{2}-k^{2})}+C_{1}e^{\sqrt{\kappa^{2}-k^{2}}s}+D_{1}e^{-\sqrt{\kappa^{2}-k^{2}}s}\\
 & f_{i}(s)=-\frac{\lambda_{i}}{2k^{2}}+\frac{C_{i}}{\sin(kl)}\cos\left[k(s-l)\right],\quad i=2,\cdots,5
\end{aligned}
\label{eq:case-I}
\end{equation}
and their derivatives are given by
\[
\begin{aligned} & f_{1}^{\prime}(s)=\sqrt{\kappa^{2}-k^{2}}\left(C_{1}e^{\sqrt{\kappa^{2}-k^{2}}s}-D_{1}e^{-\sqrt{\kappa^{2}-k^{2}}s}\right)\\
 & f_{i}^{\prime}(s)=-\frac{kC_{i}}{\sin(kl)}\sin\left[k(s-l)\right],\quad i=2,\cdots,5
\end{aligned}
\]
In the second of (\ref{eq:case-I}) it is $\sin(kl)\neq0$, for otherwise
$k=\frac{n\pi}{l}$, $n\in\mathbb{Z}$, and by $0<k<\kappa$ it follows
that $0<n<\frac{\kappa l}{\pi}$. However, by simple geometric arguments,
$\kappa l=\frac{l}{R}\leqslant\frac{\pi}{3}$.

On spine 1 we have
\[
\begin{aligned} & D_{\nu_{1}}f_{1}(0)=-f_{1}^{\prime}(0)=-\sqrt{\kappa^{2}-k^{2}}\left(C_{1}-D_{1}\right)\\
 & D_{\nu_{2}}f_{2}(0)=-f_{2}^{\prime}(0)=-kC_{2}\\
 & D_{\nu_{3}}f_{3}(0)=-f_{3}^{\prime}(0)=-kC_{3}
\end{aligned}
\]
and similarly on spine 2,
\[
\begin{aligned} & D_{\nu_{1}}f_{1}(L)=+f_{1}^{\prime}(L)=\sqrt{\kappa^{2}-k^{2}}\left(C_{1}e^{\sqrt{\kappa^{2}-k^{2}}L}-D_{1}e^{-\sqrt{\kappa^{2}-k^{2}}L}\right)\\
 & D_{\nu_{4}}f_{4}(0)=-f_{4}^{\prime}(0)=-kC_{4}\\
 & D_{\nu_{5}}f_{5}(0)=-f_{5}^{\prime}(0)=-kC_{5}
\end{aligned}
\]
For brevity we set
\[
\textrm{Case}\,\textrm{I:}\quad a=\frac{2}{\sqrt{3}}\sqrt{1-\left(\frac{k}{\kappa}\right)^{2}},\;z=e^{L\sqrt{\kappa^{2}-k^{2}}},\;b=\frac{2l}{k^{2}}-\frac{L}{\kappa^{2}-k^{2}}
\]
From the conditions (\ref{eq:eigv-sp1}), (\ref{eq:eigv-sp2}) on
spine 1 we obtain
\begin{equation}
\lambda_{2}-\lambda_{3}-2k^{2}\left[\cot(kl)-\frac{\sqrt{3}}{2}\frac{k}{\kappa}\right](C_{2}-C_{3})=0\label{eq:X-21}
\end{equation}
\begin{equation}
\frac{\lambda_{2}+\lambda_{3}}{2(\kappa^{2}-k^{2})}-\left(a+1\right)C_{1}+\frac{1}{\sqrt{3}}\frac{k}{\kappa}(C_{2}+C_{3})+\left(a-1\right)D_{1}=0\label{eq:X-22}
\end{equation}
By the compatibility condition $f_{1}+f_{2}+f_{3}=0$ on spine 1 we
obtain
\begin{equation}
-\frac{1}{2}\left(\frac{1}{\kappa^{2}-k^{2}}+\frac{1}{k^{2}}\right)(\lambda_{2}+\lambda_{3})+C_{1}+(C_{2}+C_{3})\cot(kl)+D_{1}=0\label{eq:X-23}
\end{equation}
From the conditions (\ref{eq:eigv-sp1}), (\ref{eq:eigv-sp2}) on
spine 2 we obtain
\begin{equation}
\lambda_{2}-\lambda_{3}-2k^{2}\left[\cot(kl)-\frac{\sqrt{3}}{2}\frac{k}{\kappa}\right](C_{4}-C_{5})=0\label{eq:X-24}
\end{equation}
\begin{equation}
\frac{\lambda_{2}+\lambda_{3}}{2(\kappa^{2}-k^{2})}+\left(a-1\right)zC_{1}+\frac{1}{\sqrt{3}}\frac{k}{\kappa}(C_{4}+C_{5})-\left(a+1\right)\frac{1}{z}D_{1}=0\label{eq:X-25}
\end{equation}
The equality $f_{1}+f_{2}+f_{3}=0$ on spine 2 gives
\begin{equation}
-\frac{1}{2}\left(\frac{1}{\kappa^{2}-k^{2}}+\frac{1}{k^{2}}\right)(\lambda_{2}+\lambda_{3})+zC_{1}+(C_{4}+C_{5})\cot(kl)+\frac{1}{z}D_{1}=0\label{eq:X-26}
\end{equation}
Finally the volume conservation equations (\ref{eq:100}) give the
equations
\begin{equation}
b\frac{\lambda_{2}}{2}-\frac{L}{\kappa^{2}-k^{2}}\frac{\lambda_{3}}{2}+\frac{z-1}{\sqrt{\kappa^{2}-k^{2}}}\left(C_{1}+\frac{1}{z}D_{1}\right)-\frac{1}{k}(C_{2}+C_{4})=0\label{eq:X-27}
\end{equation}
\begin{equation}
-\frac{L}{\kappa^{2}-k^{2}}\frac{\lambda_{2}}{2}+b\frac{\lambda_{3}}{2}+\frac{z-1}{\sqrt{\kappa^{2}-k^{2}}}\left(C_{1}+\frac{1}{z}D_{1}\right)-\frac{1}{k}(C_{3}+C_{5})=0\label{eq:X-28}
\end{equation}

The following lemma allows the reduction of this system to a simpler
one.
\begin{lem}
\label{lem:case-i}Assume $\left[\tan(\kappa l)+\sqrt{3}\right]\kappa L<4$.
Then the $8\times8$ linear system of equations (\ref{eq:X-21})-(\ref{eq:X-28})
is equivalent to the following $3\times3$ linear system
\[
(S_{1})\;\begin{cases}
\frac{1}{\kappa^{2}-k^{2}}\lambda_{2}+\left[(a-1)z-a-1\right]C_{1}+\frac{2}{\sqrt{3}}\frac{k}{\kappa}C_{2}=0\\
\left(\frac{1}{\kappa^{2}-k^{2}}+\frac{1}{k^{2}}\right)\lambda_{2}-(z+1)C_{1}-2C_{2}\cot(kl)=0\\
\left(\frac{l}{k^{2}}-\frac{L}{\kappa^{2}-k^{2}}\right)\lambda_{2}+2\frac{z-1}{\sqrt{\kappa^{2}-k^{2}}}C_{1}-\frac{2}{k}C_{2}=0
\end{cases}
\]
and $C_{2}=C_{3}=C_{4}=C_{5}$, $D_{1}=zC_{1}$, $\lambda_{2}=\lambda_{3}$.
\end{lem}

\begin{rem}
The condition is satisfied if $\kappa L<\sqrt{3}$, for by simple
geometric arguments (see Figure \ref{fig:DTJ}) $\kappa l=\omega<\frac{\pi}{6}$,
and $\tan(\kappa l)<\frac{1}{\sqrt{3}}$.
\end{rem}

\begin{proof}
The pairs of equations (\ref{eq:X-27}), (\ref{eq:X-28}) and (\ref{eq:X-21}),
(\ref{eq:X-24}) give
\begin{equation}
C_{2}-C_{3}+C_{4}-C_{5}=\frac{l}{k}(\lambda_{2}-\lambda_{3})\label{eq:X-56}
\end{equation}
and
\[
\left[\cot(kl)-\frac{\sqrt{3}}{2}\frac{k}{\kappa}\right](C_{2}-C_{3}-C_{4}+C_{5})=0
\]
Since $\cot(kl)\not=\frac{\sqrt{3}}{2}\frac{k}{\kappa}$, which follows
from $\frac{\sqrt{3}}{2}\frac{k}{\kappa}\tan(kl)<\frac{\sqrt{3}}{2}\tan\omega<\frac{1}{2}$
and $\omega<\frac{\pi}{6}$ (see Figure \ref{fig:DTJ}), we obtain
\begin{equation}
C_{2}-C_{3}-C_{4}+C_{5}=0\label{eq:X-57}
\end{equation}
By equations (\ref{eq:X-23}), (\ref{eq:X-26}) we obtain
\begin{equation}
C_{2}+C_{3}-C_{4}-C_{5}=(C_{1}-\frac{1}{z}D_{1})(z-1)\tan(kl)\label{eq:X-58}
\end{equation}
and by (\ref{eq:X-22}), (\ref{eq:X-25})
\begin{equation}
C_{2}+C_{3}-C_{4}-C_{5}=\sqrt{3}(C_{1}-\frac{1}{z}D_{1})\frac{\kappa}{k}\left[(a-1)z+a+1\right]\label{eq:X-59}
\end{equation}
The last two equations give
\begin{equation}
D_{1}=zC_{1}\label{eq:X-61}
\end{equation}
and
\[
\frac{k}{\kappa}\tan(kl)+\sqrt{3}\left(1-a\frac{z+1}{z-1}\right)=0
\]
which on setting $x=\frac{k}{\kappa}$, $l^{\star}=\kappa l$, $L^{\star}=\kappa L$
and using the expression of $a$, assumes the form
\begin{equation}
x\tan(l^{\star}x)=2\sqrt{1-x^{2}}\frac{z+1}{z-1}-\sqrt{3}.\label{eq:106}
\end{equation}
The function on the right side is decreasing and thus attains its
minimum at $x=1$, hence it is greater than $\frac{4}{L^{\star}}-\sqrt{3}$.
From $x\tan(l^{\star}x)<\tan(l^{\star})=\tan(\kappa l)$ and the hypothesis,
it follows that equation (\ref{eq:106}) has no solution.

By (\ref{eq:X-59}), (\ref{eq:X-61}) and (\ref{eq:X-56}), (\ref{eq:X-57})
we obtain
\begin{equation}
C_{2}+C_{3}-C_{4}-C_{5}=0\label{eq:X-63}
\end{equation}
\begin{equation}
C_{2}-C_{3}=C_{4}-C_{5}=\frac{l}{2k}(\lambda_{2}-\lambda_{3})\label{eq:X-64-65}
\end{equation}
By (\ref{eq:X-21}) and (\ref{eq:X-64-65}) it follows that
\[
(\lambda_{2}-\lambda_{3})\left[kl\left(\cot(kl)-\frac{\sqrt{3}}{2}\frac{k}{\kappa}\right)-1\right]=0
\]
and from this and $\cot(kl)-\frac{\sqrt{3}}{2}\frac{k}{\kappa}<0$,
$\lambda_{2}=\lambda_{3}$. Use of equations (\ref{eq:X-63}), (\ref{eq:X-64-65}),
and the remaining equations of system (\ref{eq:X-21})-(\ref{eq:X-28}),
i.e. (\ref{eq:X-22}), (\ref{eq:X-23}) and (\ref{eq:X-27}), completes
the proof. 
\end{proof}

\subsection{Case II: $\mu<-\kappa^{2},\;k>0$}

Here $k$ is defined by $k^{2}=-(\mu+\kappa^{2})$, so that the valid
range of $k$ is $k>0$. In this case $k_{i}=k$ for $i=2,\cdots,5$
and $k_{1}=-(\mu+\kappa_{1}^{2})=-\mu>0$. Thus we have a solution
type (II) for all $f_{i}$,
\begin{equation}
\begin{aligned} & f_{1}(s)=-\frac{\lambda_{2}+\lambda_{3}}{2(k^{2}+\kappa^{2})}+C_{1}e^{\sqrt{k^{2}+\kappa^{2}}s}+D_{1}e^{-\sqrt{k^{2}+\kappa^{2}}s}\\
 & f_{i}(s)=\frac{\lambda_{i}}{2k^{2}}+2C_{i}e^{kl}\cosh k(s-l),\quad i=2,\cdots,5
\end{aligned}
\label{eq:case-II}
\end{equation}
and their derivatives are given by
\[
\begin{aligned} & f_{1}^{\prime}(s)=\sqrt{k^{2}+\kappa^{2}}\left(C_{1}e^{\sqrt{k^{2}+\kappa^{2}}s}-D_{1}e^{-\sqrt{k^{2}+\kappa^{2}}s}\right)\\
 & f_{i}^{\prime}(s)=2kC_{i}e^{kl}\sinh k(s-l),\quad i=2,\cdots,5
\end{aligned}
\]
Proceeding as in case I we obtain the following linear system:
\begin{equation}
\frac{\lambda_{2}-\lambda_{3}}{2k^{2}}+\left[\left(\frac{\sqrt{3}}{2}\frac{k}{\kappa}+1\right)e^{2kl}-\left(\frac{\sqrt{3}}{2}\frac{k}{\kappa}-1\right)\right](C_{2}-C_{3})=0\label{eq:X-29}
\end{equation}
\begin{equation}
\frac{\lambda_{2}-\lambda_{3}}{2(k^{2}+\kappa^{2})}-(a+1)C_{1}-\frac{1}{\sqrt{3}}\frac{k}{\kappa}(e^{2kl}-1)(C_{2}+C_{3})+(a-1)D_{1}=0\label{eq:X-30}
\end{equation}
\begin{equation}
\frac{\lambda_{2}-\lambda_{3}}{2}\left(\frac{1}{k^{2}}-\frac{1}{k^{2}+\kappa^{2}}\right)+C_{1}+(e^{2kl}+1)(C_{2}+C_{3})+D_{1}=0\label{eq:X-31}
\end{equation}
\begin{equation}
\frac{\lambda_{2}+\lambda_{3}}{2k^{2}}+\left[\left(\frac{\sqrt{3}}{2}\frac{k}{\kappa}+1\right)e^{2kl}-\left(\frac{\sqrt{3}}{2}\frac{k}{\kappa}-1\right)\right](C_{4}-C_{5})=0\label{eq:X-32}
\end{equation}
\begin{equation}
\frac{\lambda_{2}-\lambda_{3}}{2(k^{2}+\kappa^{2})}+(a-1)zC_{1}-\frac{1}{\sqrt{3}}\frac{k}{\kappa}(e^{2kl}-1)(C_{4}+C_{5})-(a+1)\frac{1}{z}D_{1}=0\label{eq:X-33}
\end{equation}
\begin{equation}
\frac{\lambda_{2}+\lambda_{3}}{2}\left(\frac{1}{k^{2}}-\frac{1}{k^{2}+\kappa^{2}}\right)+zC_{1}+(e^{2kl}+1)(C_{4}+C_{5})+\frac{1}{z}D_{1}=0\label{eq:X-34}
\end{equation}
\begin{equation}
-b\frac{\lambda_{2}}{2}-\frac{L}{k^{2}+\kappa^{2}}\frac{\lambda_{3}}{2}+\frac{z-1}{\sqrt{k^{2}+\kappa^{2}}}\left(C_{1}+\frac{1}{z}D_{1}\right)-\frac{e^{2kl}-1}{k}(C_{2}+C_{4})=0\label{eq:X-35}
\end{equation}
\begin{equation}
-\frac{L}{k^{2}+\kappa^{2}}\frac{\lambda_{2}}{2}-b\frac{\lambda_{3}}{2}+\frac{z-1}{\sqrt{k^{2}+\kappa^{2}}}\left(C_{1}+\frac{1}{z}D_{1}\right)-\frac{e^{2kl}-1}{k}(C_{3}+C_{5})=0\label{eq:X-36}
\end{equation}
The $a$, $b$, $z$ are now defined as
\[
\textrm{Case}\,\textrm{II:}\quad a=\frac{2}{\sqrt{3}}\sqrt{1+\left(\frac{k}{\kappa}\right)^{2}},\;z=e^{L\sqrt{k^{2}+\kappa^{2}}},\;b=\frac{2l}{k^{2}}+\frac{L}{k^{2}+\kappa^{2}}
\]

As previously,
\begin{lem}
\label{lem:case-ii}The $8\times8$ linear system of equations (\ref{eq:X-29})-(\ref{eq:X-36})
is equivalent to the following $3\times3$ linear system
\[
(S_{2})\;\begin{cases}
\frac{1}{k^{2}+\kappa^{2}}\lambda_{2}-\left[z+1-a(z-1)\right]C_{1}-\frac{2}{\sqrt{3}}\frac{k}{\kappa}\left(e^{2kl}-1\right)C_{2}=0\\
\left(\frac{1}{k^{2}}-\frac{1}{k^{2}+\kappa^{2}}\right)\lambda_{2}+(z+1)C_{1}+2C_{2}\left(e^{2kl}+1\right)=0\\
\left(\frac{l}{k^{2}}+\frac{L}{k^{2}+\kappa^{2}}\right)\lambda_{2}-2\frac{z-1}{\sqrt{k^{2}+\kappa^{2}}}C_{1}+\frac{2}{k}\left(e^{2kl}-1\right)C_{2}=0
\end{cases}
\]
and $C_{2}=C_{3}=C_{4}=C_{5}$, $D_{1}=zC_{1}$, $\lambda_{2}=\lambda_{3}$.
\end{lem}

The proof of Lemma \ref{lem:case-ii} is analogous to the proof of
Lemma \ref{lem:case-i}.

\subsection{Case III: $\mu=-\kappa^{2}$}

In this case $k_{1}^{2}=-(\mu+\kappa_{1}^{2})=-\mu=\kappa^{2}$, and
$k_{i}=\mu+\kappa_{i}^{2}=\mu+\kappa^{2}=0$ $(i=2,\cdots,5)$, and
thus we have a solution type (II) for $f_{1}$ and type (III) for
all other $f_{i}$:
\begin{equation}
\begin{aligned} & f_{1}(s)=-\frac{\lambda_{2}+\lambda_{3}}{2\kappa^{2}}+C_{1}e^{\kappa s}+D_{1}e^{-\kappa s}\\
 & f_{i}(s)=-\frac{\lambda_{i}}{4}(s-l)^{2}+C_{i},\quad i=2,\cdots,5
\end{aligned}
\label{eq:case-III}
\end{equation}
Their derivatives are
\[
\begin{aligned} & f_{1}^{\prime}(s)=\kappa\left(C_{1}e^{\kappa s}-D_{1}e^{-\kappa s}\right)\\
 & f_{i}^{\prime}(s)=-\frac{\lambda_{i}}{2}(s-l),\quad i=2,\cdots,5
\end{aligned}
\]
As previously we obtain the system
\begin{equation}
\left(\frac{\sqrt{3}}{\kappa}+l\right)\frac{\lambda_{2}+\lambda_{3}}{4}-\kappa\left(1+\frac{\sqrt{3}}{2}\right)C_{1}+\kappa\left(1-\frac{\sqrt{3}}{2}\right)D_{1}=0\label{eq:X-37}
\end{equation}
\begin{equation}
\left(\frac{\sqrt{3}}{\kappa}+l\right)l\frac{\lambda_{3}-\lambda_{2}}{4}+C_{2}-C_{3}=0\label{eq:X-38}
\end{equation}
\begin{equation}
-\left(\frac{2}{\kappa^{2}}+l^{2}\right)\frac{\lambda_{2}+\lambda_{3}}{4}+C_{1}+C_{2}+C_{3}+D_{1}=0\label{eq:X-39}
\end{equation}
\begin{equation}
\left(\frac{\sqrt{3}}{\kappa}+l\right)\frac{\lambda_{2}+\lambda_{3}}{4}+\left(1-\frac{\sqrt{3}}{2}\right)\kappa e^{\kappa L}C_{1}-\left(1+\frac{\sqrt{3}}{2}\right)\kappa e^{-\kappa L}D_{1}=0\label{eq:X-40}
\end{equation}
\begin{equation}
\left(\frac{\sqrt{3}}{\kappa}+l\right)l\frac{\lambda_{3}-\lambda_{2}}{4}+C_{4}-C_{5}=0\label{eq:X-41}
\end{equation}
\begin{equation}
-\left(\frac{2}{\kappa^{2}}+l^{2}\right)\frac{\lambda_{2}+\lambda_{3}}{4}+e^{\kappa L}C_{1}+C_{4}+C_{5}+e^{-\kappa L}D_{1}=0\label{eq:X-42}
\end{equation}
\begin{equation}
\left(\frac{L}{\kappa^{2}}-\frac{l^{3}}{3}\right)\frac{\lambda_{2}}{2}+\frac{L}{\kappa^{2}}\frac{\lambda_{3}}{2}-\frac{e^{\kappa L}-1}{\kappa}(C_{1}+e^{-\kappa L}D_{1})+l(C_{2}+C_{4})=0\label{eq:X-43}
\end{equation}
\begin{equation}
\frac{L}{\kappa^{2}}\frac{\lambda_{2}}{2}+\left(\frac{L}{\kappa^{2}}-\frac{l^{3}}{3}\right)\frac{\lambda_{3}}{2}-\frac{e^{\kappa L}-1}{\kappa}(C_{1}+e^{-\kappa L}D_{1})+l(C_{3}+C_{5})=0\label{eq:X-44}
\end{equation}
\begin{lem}
\label{lem:case-iii}The $8\times8$ linear system of equations (\ref{eq:X-37})-(\ref{eq:X-44})
is equivalent to the following $3\times3$ linear system
\[
(S_{3})\;\begin{cases}
\left(\sqrt{3}+\kappa l\right)\lambda_{2}^{\star}+\left[z-1-\tfrac{\sqrt{3}}{2}(z+1)\right]C_{1}=0\\
-\left(\kappa^{2}l^{2}+2\right)\lambda_{2}^{\star}+(z+1)C_{1}+2C_{2}=0\\
\left(\kappa L-\tfrac{1}{6}\kappa^{3}l^{3}\right)\lambda_{2}^{\star}-(z-1)C_{1}+\kappa lC_{2}=0
\end{cases}
\]
and $C_{2}=C_{3}=C_{4}=C_{5}$, $D_{1}=zC_{1}$, $\lambda_{2}=\lambda_{3}$.
In $(S_{3})$, $\lambda_{2}^{\star}=\frac{\lambda_{2}}{2\kappa^{2}}$,
and $z=e^{\kappa L}$.
\end{lem}

The proof is analogous to that of Lemma \ref{lem:case-i}.

\subsection{Case IV: $\mu=0$}

This is the case of neutral stability. Since $k_{1}^{2}=-(\mu+\kappa_{1}^{2})=0$,
and $k_{i}=\mu+\kappa_{i}^{2}=\kappa^{2}>0$ $(i=2,\cdots,5)$, we
have a solution type (III) for $f_{1}$ and type (I) for all other
$f_{i}$:
\begin{equation}
\begin{aligned} & f_{1}(s)=\frac{\lambda_{2}+\lambda_{3}}{4}s^{2}+C_{1}s+D_{1}\\
 & f_{i}(s)=-\frac{\lambda_{i}}{2\kappa^{2}}+\frac{C_{i}}{\sin(\kappa l)}\cos\kappa(s-l),\quad i=2,\cdots,5
\end{aligned}
\label{eq:case-IV}
\end{equation}
Proceeding as in case I we obtain the following linear system:
\begin{equation}
\frac{2}{\kappa}C_{1}-C_{2}-C_{3}+\sqrt{3}D_{1}=0\label{eq:X-IV-1}
\end{equation}
\begin{equation}
-\frac{\lambda_{2}-\lambda_{3}}{2\kappa^{2}}+(C_{2}-C_{3})\left(\cot\kappa l-\tfrac{\sqrt{3}}{2}\right)=0\label{eq:X-IV-2}
\end{equation}
\begin{equation}
-\frac{\lambda_{2}+\lambda_{3}}{2\kappa^{2}}+(C_{2}+C_{3})\cot\kappa l+D_{1}=0\label{eq:X-IV-3}
\end{equation}
\begin{equation}
\left(\tfrac{\sqrt{3}}{2}\kappa L-1\right)L\frac{\lambda_{2}+\lambda_{3}}{2\kappa}+\left(\tfrac{\sqrt{3}}{2}\kappa L-1\right)\frac{C_{1}}{\kappa}-\tfrac{1}{2}(C_{4}+C_{5})+\tfrac{\sqrt{3}}{2}D_{1}=0\label{eq:X-IV-4}
\end{equation}
\begin{equation}
-\frac{\lambda_{2}-\lambda_{3}}{2\kappa^{2}}+(C_{4}-C_{5})\left(\cot\kappa l-\tfrac{\sqrt{3}}{2}\right)=0\label{eq:X-IV-5}
\end{equation}
\begin{equation}
\frac{\lambda_{2}+\lambda_{3}}{2\kappa^{2}}\left(\kappa^{2}L^{2}-1\right)+LC_{1}+(C_{4}+C_{5})\cot\kappa l+D_{1}=0\label{eq:X-IV-6}
\end{equation}
\begin{equation}
\left(\frac{\kappa^{3}L^{3}}{12}+\kappa l\right)\frac{\lambda_{2}}{\kappa^{2}}+\frac{\kappa^{3}L^{3}}{12}\frac{\lambda_{3}}{\kappa^{2}}+\frac{\kappa L^{2}}{2}C_{1}-(C_{2}+C_{4})+\kappa LD_{1}=0\label{eq:X-IV-7}
\end{equation}
\begin{equation}
\frac{\kappa^{3}L^{3}}{12}\frac{\lambda_{2}}{\kappa^{2}}+\left(\frac{\kappa^{3}L^{3}}{12}+\kappa l\right)\frac{\lambda_{3}}{\kappa^{2}}+\frac{\kappa L^{2}}{2}C_{1}-(C_{3}+C_{5})+\kappa LD_{1}=0\label{eq:X-IV-8}
\end{equation}
\begin{lem}
\label{lem:case-iv}The linear system of equations (\ref{eq:X-IV-1})-(\ref{eq:X-IV-8})
has a nontrivial solution if and only if the following condition is
satisfied:
\begin{equation}
\begin{gathered}\left[\phi\left(\tfrac{1}{6}L^{\star3}-\tfrac{\sqrt{3}}{2}L^{\star2}+L^{\star}+l\right)-\left(\sqrt{3}-L^{\star}\right)\right]\left(\phi L^{\star}-4\cot l^{\star}\right)+\\
\tfrac{1}{2}\left(4-\phi L^{\star}\right)\left(\sqrt{3}-L^{\star}\right)L^{\star}\left(\phi L^{\star}-2\cot l^{\star}\right)=0
\end{gathered}
\label{eq:107}
\end{equation}
In (\ref{eq:107}) $L^{\star}=\kappa L$, $l^{\star}=\kappa l$, and
$\phi=\sqrt{3}\cot l^{\star}+1$.
\end{lem}

\begin{proof}
Equations (\ref{eq:X-IV-7}), (\ref{eq:X-IV-8}), (\ref{eq:X-IV-2})
and (\ref{eq:X-IV-5}) are equivalent to the following four equations:
\[
\lambda_{2}=\lambda_{3},\quad C_{2}=C_{3},\quad C_{4}=C_{5}
\]
and
\begin{equation}
\left(\tfrac{1}{6}L^{\star3}+l^{\star}\right)\lambda_{2}^{\star}+\tfrac{1}{2}L^{\star2}C_{1}^{\star}-(C_{2}+C_{4})+L^{\star}D_{1}=0\label{eq:X-IV-7p}
\end{equation}
We have switched to the dimensionless quantities $L^{\star}$, $l^{\star}$,
$\lambda_{2}^{\star}=\frac{\lambda_{2}}{\kappa^{2}}$ and $C_{2}^{\star}=\frac{C_{2}}{\kappa}$.
Similarly, equations (\ref{eq:X-IV-1}), (\ref{eq:X-IV-4}), (\ref{eq:X-IV-3})
and (\ref{eq:X-IV-6}) are equivalent to
\begin{equation}
C_{2}-\tfrac{\sqrt{3}}{2}D_{1}=C_{1}^{\star},\quad2C_{2}\cot l^{\star}+D_{1}=\lambda_{2}^{\star}\label{eq:X-IV-1p-3p}
\end{equation}
and
\begin{equation}
C_{2}-C_{4}=-L^{\star}\left(\tfrac{\sqrt{3}}{2}L^{\star}-1\right)\lambda_{2}^{\star}-\left(\tfrac{\sqrt{3}}{2}L^{\star}-2\right)C_{1}^{\star}\label{eq:X-IV-11p}
\end{equation}
\begin{equation}
\left[\left(\tfrac{\sqrt{3}}{2}L^{\star}-1\right)\cot l^{\star}+\tfrac{1}{2}L^{\star}\right]L^{\star}\lambda_{2}^{\star}=\left[\left(\tfrac{\sqrt{3}}{2}L^{\star}-2\right)\cot l^{\star}+\tfrac{1}{2}L^{\star}\right]C_{1}^{\star}\label{eq:X-IV-13}
\end{equation}
By (\ref{eq:X-IV-1p-3p}), solving for $C_{2}$ and $D_{1}$ in terms
of $C_{1}^{\star}$, $\lambda_{2}^{\star}$, and then solving (\ref{eq:X-IV-7p})
for $C_{4}$ again in terms of $C_{1}^{\star}$, $\lambda_{2}^{\star}$,
and substituting in (\ref{eq:X-IV-11p}) gives a linear homogeneous
equation in $C_{1}^{\star}$ and $\lambda_{2}^{\star}$. The compatibility
condition of the system comprised of this equation and (\ref{eq:X-IV-13})
gives equation (\ref{eq:107}). 
\end{proof}

\subsection{Existence of stable disconnected partitions}

In the following theorem we state the example announced at the beginning
of this section, showing the existence of stable disconnected three
phase partitionings by triple junction systems.
\begin{thm}
\label{thm:stab-discon-tj}Let $\Omega$ be a convex domain in $\mathbb{R}^{2}$,
and $T=(M_{1},\cdots,M_{5})$ a minimal disconnected three-phase partitioning
of $\Omega$ by a system of two $C^{2}$ triple junctions as in Figure
\ref{fig:DTJ}, with volume constraints. Furthermore, for $\Omega$
and the partitioning system $T$ we make the following assumptions:

\emph{(H1)} The boundary $\Sigma=\partial\Omega$ is $C^{2}$ in a
neighborhood of $\Sigma\cap\overline{T}$ and it is flat at $\overline{T}\cap\Sigma$.
In particular this means $\sigma=0$ at all points of $\overline{T}\cap\Sigma$.

\emph{(H2)} $M_{1}$ is flat, i.e. $\kappa_{1}=0$, and the length
of $M_{1}$ is $L$.

\emph{(H3) }All\emph{ }other leaves have the same curvature $\kappa\neq0$
and the same length $|M_{i}|=l$, $i=2,\cdots,5$.

\emph{(H4)} $\alpha<0$ in the orientation of Figure \ref{fig:DTJ}.

Then there is a $L_{0}>0$, possibly depending on $l$ and $\kappa$,
such that for $L\leqslant L_{0}$ the disconnected triple junction
partitioning $T$ is stable.
\end{thm}

\begin{proof}
We will prove that the cases (I)-(IV) give no eigenvalue $\mu$, and
thus the minimal eigenvalue of the problem is necessarily positive,
which then by Proposition \ref{prop:Categ-b} proves the assertion.

Assume $\left(\tan l^{\star}+\sqrt{3}\right)L_{0}^{\star}<4$. The
possible eigenvalues in the range $-\kappa^{2}<\mu<0$ (case I) are
given by the solution of the equation $D_{1}(x)=0$ in $0<x<1$, where
$D_{1}$ is the determinant of $(S_{1})$ (see Lemma \ref{lem:case-i})
multiplied by $\kappa k(k+\kappa)$, 
\begin{equation}
D_{1}(x)=\frac{1}{x^{2}(1-x)}\left|\begin{array}{ccc}
x^{2} & a(z-1)-(z+1) & \tfrac{2}{\sqrt{3}}x^{2}\\
1 & -(z+1) & -2x\cot(l^{\star}x)\\
l^{\star}(1-x^{2})-L^{\star}x^{2} & 2\frac{z-1}{\sqrt{1-x^{2}}} & -2
\end{array}\right|,
\end{equation}
$k^{2}=\mu+\kappa^{2}$, $x=\frac{k}{\kappa}$, $0<x<1$, $a=\frac{2}{\sqrt{3}}\sqrt{1-x^{2}}$,
and $z=e^{L^{\star}\sqrt{1-x^{2}}}$. Clearly $D_{1}$ is real analytic
in $L^{\star}$, and
\begin{equation}
D_{1}(x)=4\frac{x+1}{x^{2}}\left[\frac{l^{\star}}{\sqrt{3}}x^{2}+l^{\star}x\cot\left(l^{\star}x\right)-1\right]+\mathcal{O}(L^{\star})\label{eq:D1}
\end{equation}
We will prove that the term inside square brackets is $\geq C_{0}$,
where $C_{0}>0$ is a constant, in the interval $0\leqslant x\leqslant1$.
Considering the function
\[
f(x)=\frac{l^{\star}}{\sqrt{3}}x^{2}\sin(l^{\star}x)+l^{\star}x\cos\left(l^{\star}x\right)-\sin(l^{\star}x)
\]
with derivative
\[
f^{\prime}(x)=l^{\star}x\sin\left(l^{\star}x\right)\left[\tfrac{2}{\sqrt{3}}-l^{\star}+\tfrac{1}{\sqrt{3}}l^{\star}x\cot\left(l^{\star}x\right)\right]>0,\;x>0
\]
we obtain $f(x)>0$ for $x>0$ and
\[
\frac{l^{\star}}{\sqrt{3}}x^{2}+l^{\star}x\cot\left(l^{\star}x\right)-1>0
\]
The limits as $x\to0+,1-$ of the 0-th order term in the expansion
(\ref{eq:D1}) are positive in the valid range of $l^{\star}$, $]0,\frac{\pi}{6}[$.
Furthermore, the function $\frac{\partial D_{1}}{\partial L^{\star}}$
is bounded and continuous in $[0,1]$. Application of Taylor's formula
with remainder yields a $L_{1}^{\star}>0$ such that $L_{1}^{\star}\leqslant L_{0}^{\star}$
and $D_{1}(x)\not=0$ in $]0,1[$ for all $0<L^{\star}\leqslant L_{1}^{\star}$.

The possible eigenvalues in the range $\mu<-\kappa^{2}$ are given
by the solution of the equation $D_{2}(x)=0$ in $]0,\infty[$, where
$D_{2}$ is the determinant of $(S_{2})$ (see Lemma \ref{lem:case-ii})
multiplied by $\kappa k(k+\kappa)$,
\[
D_{2}(x)=\frac{1}{x^{2}}\left|\begin{array}{ccc}
x^{2} & a(z-1)-(z+1) & -\frac{2}{\sqrt{3}}x\left(e^{2xl^{\star}}-1\right)\\
1 & z+1 & 2\left(e^{2xl^{\star}}+1\right)\\
L^{\star}x^{2}+l^{\star}(x^{2}+1) & -2\frac{z-1}{\sqrt{x^{2}+1}} & \frac{2}{x}\left(e^{2xl^{\star}}-1\right)
\end{array}\right|,
\]
$k^{2}=-(\mu+\kappa^{2})$, $x=\frac{k}{\kappa}$, $x>0$, $a=\frac{2}{\sqrt{3}}\sqrt{1+x^{2}}$,
and $z=e^{L^{\star}\sqrt{1+x^{2}}}$. We have
\[
D_{2}(x)=ze^{2xl^{\star}}x\left[2\sqrt{3}(L^{\star}+l^{\star})+\mathcal{O}(\frac{1}{x})\right]
\]
as $x\to+\infty$. Consequently, we can select a $x_{0}>0$ (which
is independent of $L^{\star}$) such that $D_{2}(x)\neq0$ for $x\geqslant x_{0}$.
To prove that $D_{2}$ has no roots in $]0,x_{0}[$ we consider its
Taylor expansion in $L^{\star}$,
\[
D_{2}(x)=4\frac{x^{2}+1}{x^{3}}\left[\left(\tfrac{l^{\star}}{\sqrt{3}}x^{2}-l^{\star}x+1\right)e^{2l^{\star}x}-\tfrac{l^{\star}}{\sqrt{3}}x^{2}-l^{\star}x-1\right]+\mathcal{O}(L^{\star})
\]
We will prove that the 0-th order term is positive in $[0,x_{0}]$.
To this purpose we write the term inside the square brackets in the
form
\[
\left(\tfrac{1}{\sqrt{3}}l^{\star}x^{2}-l^{\star}x+1\right)\left(e^{2l^{\star}x}-1\right)-2l^{\star}x
\]
and apply the inequality $e^{t}-1\geqslant t+\frac{1}{2}t^{2}$ with
$t=2l^{\star}x$. In this way we obtain
\[
\left(\tfrac{l^{\star}}{\sqrt{3}}x^{2}-l^{\star}x+1\right)e^{2l^{\star}x}-\tfrac{l^{\star}}{\sqrt{3}}x^{2}-l^{\star}x-1>\tfrac{2}{\sqrt{3}}l^{\star2}x^{3}\left(l^{\star}x-\sqrt{3}l^{\star}+1\right).
\]
Since $l^{\star}\leqslant\frac{\pi}{6}$, we have $l^{\star}x-\sqrt{3}l^{\star}+1\geqslant1-\frac{\sqrt{3}\pi}{6}>0$.
This proves the existence of a $C_{0}>0$ such that
\[
4\frac{x^{2}+1}{x^{3}}\left[\left(\tfrac{l^{\star}}{\sqrt{3}}x^{2}-l^{\star}x+1\right)e^{2l^{\star}x}-\tfrac{l^{\star}}{\sqrt{3}}x^{2}-l^{\star}x-1\right]\geqslant C_{0},\;x>0.
\]
The partial derivative of $D_{2}$ with respect to $L^{\star}$ is
a bounded continuous function in $[0,x_{0}]$. By Taylor's formula
with remainder, we can select $L_{2}^{\star}\leqslant L_{1}^{\star}$
so small that $|\mathcal{O}(L^{\star})|<\frac{C_{0}}{2}$ for $L^{\star}\leqslant L_{2}^{\star}$.
Then $D_{2}(x)>\frac{C_{0}}{2}$ in $]0,x_{0}[$, and this completes
the proof that $D_{2}$ has no roots in $]0,+\infty[$ for $L^{\star}\leqslant L_{2}^{\star}$.

We proceed to case (III) for the eigenvalue $\mu=-\kappa^{2}$. Solving
the last two equations of $(S_{3})$ for $C_{1}$, $C_{2}$ in terms
of $\lambda_{2}^{\star}$ and substituting in the first of $(S_{3})$
gives the following necessary and sufficient condition for $(S_{3})$
to have nontrivial solutions:
\begin{equation}
\left(\sqrt{3}+2l^{\star}+\tfrac{1}{3}l^{\star3}+L^{\star}\right)(z-1)+\tfrac{1}{2}\left(l^{\star2}-\tfrac{1}{\sqrt{3}}l^{\star3}-\sqrt{3}L^{\star}\right)(z+1)=0\label{eq:109}
\end{equation}
In the limit $L\to0$ this reduces to $l^{\star}=\sqrt{3}>\frac{\pi}{6}$,
which is absurd. Hence, there is a $L_{3}^{\star}>0$ such that $L_{3}^{\star}\leqslant L_{2}^{\star}$
and equation (\ref{eq:109}) has no solution for $L^{\star}\leqslant L_{3}^{\star}$.

Finally, we treat the neutral stability case (IV), $\mu=0$. By Lemma
\ref{lem:case-iv} the eigenvalue 0 is possible only for pairs $L^{\star}$,
$l^{\star}$ satisfying equation (\ref{eq:107}). As $L^{\star}\to0$
this reduces to 
\[
-4\cot l^{\star}\left[l^{\star}\left(\sqrt{3}\cot l^{\star}+1\right)-\sqrt{3}\right]=0
\]
which, as it is easily seen, has no solution. This implies the existence
of a $L_{4}^{\star}>0$ such that $L_{4}^{\star}\leqslant L_{3}^{\star}$
and there is no $l^{\star}$ satisfying equation (\ref{eq:107}) for
all $L^{\star}\leqslant L_{4}^{\star}$. Redefinition of $L_{0}^{\star}$
as $L_{4}^{\star}$ proves the theorem.
\end{proof}

\end{document}